\documentclass[reqno,12pt]{amsart}
\usepackage{graphicx}
\usepackage{epsfig}
\usepackage{amssymb,amsthm,amsmath,amsfonts,amstext,mathrsfs,fancybox}
\usepackage{enumerate}
\usepackage{latexsym}
\usepackage{url}
\usepackage{fancyhdr}
\usepackage{color}
\usepackage[unicode]{hyperref}

\hypersetup{colorlinks=true,
linkcolor=blue,
citecolor=blue,
urlcolor=blue,
filecolor=blue,
bookmarksnumbered=true,
pdfstartview=FitH,
pdfhighlight=/N}

\newtheorem{thm}{Theorem}
\newtheorem{cor}{Corollary}

\newtheorem{prop}{Proposition}
\newtheorem{prob}{Problem}
\newtheorem{defin}{Definition}

\input xy
\xyoption{all}

\def\d{\,{\rm{d}}}

\def\s{\,{\rm{sp}}}
\def\c{\,{\rm{cp}}}
\def\sm{\,{\rm{sm}}}
\def\cm{\,{\rm{cm}}}
\def\v{\,{\varsigma}}
\def\w{\,{\mathfrak{w}}}
\newcommand{\m}{\mathbf}
\newcommand{\bl}{\bullet}

\title[The projective translation equation]
{The projective translation equation
and unramified $2$-dimensional flows with rational vector fields}
\author[G. Alkauskas]{Giedrius Alkauskas}
\address{Vilnius University, Department of Mathematics and Informatics, Naugarduko 24, LT-03225 Vilnius, Lithuania}
\email{giedrius.alkauskas@gmail.com}

\begin{document}
\begin{abstract}
Let $\m{x}=(x,y)$. Previously we have found all rational solutions of the $2$-dimensional projective translation equation, or PrTE,  $(1-z)\phi(\m{x})=\phi(\phi(\m{x}z)(1-z)/z)$; here $\phi(\m{x})=(u(x,y),v(x,y))$ is a pair of two (real or complex) functions. Solutions of this functional equation are called \emph{projective flows}. A vector field of a rational flow is a pair of $2$-homogenic rational functions. On the other hand, only special pairs of $2$-homogenic rational functions give rise to rational flows. In this paper we are interested in all non-singular (satisfying the boundary condition) and unramified (without branching points, i.e. single-valued functions in $\mathbb{C}^{2}\setminus\{\text{union of curves}\}$) projective flows whose vector field is still rational. If an orbit of the flow is given by homogeneous rational function of degree $N$, then $N$ is called the level of the flow. We prove that, up to conjugation with $1$-homogenic birational plane transformation, these are of $6$ types: 1) the identity flow; 2) one flow for each non-negative integer $N$ - these flows are rational of level $N$; 3) the level $1$ exponential flow, which is also conjugate to the level $1$ tangent flow; 4) the level $3$ flow expressable in terms of Dixonian (equianharmonic) elliptic functions; 5) the level $4$ flow expressable in terms of lemniscatic elliptic functions; 6) the level $6$ flow expressable in terms of Dixonian elliptic functions again. This reveals another aspect of the PrTE: in the latter four cases this equation is equivalent and provides a uniform framework to addition formulas for exponential, tangent, or special elliptic functions (also addition formulas for polynomials and the logarithm, though the latter appears only in branched flows). Moreover, the PrTE turns out to have a connection with P\'{o}lya-Eggenberger urn models. Another purpose of this study is expository, and we provide the list of open problems and directions in the theory of PrTE; for example, we define the notion of quasi-rational projective flows which includes curves of arbitrary genus. This study, though seemingly analytic, is in fact algebraic and our method reduces to algebraic transformations in various quotient rings of rational function fields.
\end{abstract}
\pagestyle{fancy}
\fancyhead{}
\fancyhead[LE]{{\sc Projective translation equation}}
\fancyhead[RO]{{\sc G. Alkauskas}}
\fancyhead[CE,CO]{\thepage}
\fancyfoot{}
\date{September 15, 2012}
\thanks{The author gratefully acknowledges support by the Lithuanian Science Council whose postdoctoral fellowship
 is being funded by European Union Structural Funds project ``Postdoctoral Fellowship Implementation in Lithuania".}

\maketitle
\setcounter{tocdepth}{1} \tableofcontents
\section{Introduction}
\label{sec:intro}
For convenience reasons, we write $F(x,y)\bl G(x,y)$ instead of
$\big{(}F(x,y),G(x,y)\big{)}$. We also write $\m{x}=x\bl y$. This paper is a continuation of \cite{alkauskas1,alkauskas2}. It is completely independent from the first, and also mostly independent from the second paper, apart from the (\cite{alkauskas2}, Section 4, steps I and III), which are crucial for the current work. The main needed steps will be summarized here in the Subsection \ref{sub-quadratic}. The current paper is an introduction to the prospective paper \cite{alkauskas4} where general plane vector fields are treated in the framework of quasi-rational flows; see the Subsection \ref{sub-qf}.
\subsection{Background}
\label{back}
The general ``affine" translation equation is a functional equation of the type
\begin{eqnarray*}
F(\m{x},s+t)=F(F(\m{x},s),t)\quad \m{x}\in\mathbb{C}^{n},s,t\in\mathbb{C},
\end{eqnarray*}
where $F:\mathbb{C}^{n}\times\mathbb{C}\mapsto\mathbb{C}^{n}$. Note that this is not the most general form of the translation equation; to get the feeling of the variety of structure and methods that underlies this equation, the reader may consult \cite{aczel,moszner1,moszner2,r-r1,r-r2}. Our research concentrates on the special case of this equation, where $F(\m{x},t)$ is of the form $\phi(\m{x}t)t^{-1}$; this choice is not accidental, since it exhibits several fascinating features not encountered in the general affine solutions (``affine flows").  This special case was first introduced in \cite{alkauskas1} where we considered the $2$-dimensional equation from the topologic point of view in the simplest case of the sphere $\mathbb{S}^{2}$. \\

In the paper \cite{alkauskas2} we solved the following problem: find all rational solutions of the $2$-dimensional projective translation equation (PrTE for short)
\begin{eqnarray}\setlength{\shadowsize}{2pt}\shadowbox{$\displaystyle{\quad
(1-z)\phi(\m{x})=\phi\Big{(}\phi(\m{x}z)\frac{1-z}{z}\Big{)}}$.\quad}\label{funkk}
\end{eqnarray}
Here $\phi(x,y)=u(x,y)\bl v(x,y)$ is a pair of rational functions in two real or complex variables. It appears that, up to conjugation with a $1$-homogenic birational plane transformation ($1$-BIR in short, for the structure of these see the Appendix in \cite{alkauskas2}), all rational solutions of this equation are as follows: the zero flow $0\bl 0$, two singular flows $0\bl y$ and $0\bl\frac{y}{y+1}$, an identity flow $x\bl y$, and one non-singular flow for each non-negative integer $N$, called \emph{the level} of the flow; this solution is given by $\phi_{N}$ below. The non-singular solution of this equation is called \emph{the projective flow}. \emph{Non-singularity} means that a flow satisfies the boundary condition
\begin{eqnarray}
\lim\limits_{z\rightarrow 0}\frac{\phi(\m{x}z)}{z}=\m{x}.
\label{bound}
\end{eqnarray}
Thus, each rational projective flow $u(x,y)\bl v(x,y)$ is automatically a birational plane transformation whose inverse is given by $-u(-x,-y)\bl -v(-x,-y)$.\\

The main object associated with a solution is its \emph{vector field} given by
\begin{eqnarray}
\mathbf{v}(\phi;x,y)=\varpi(x,y)\bl\varrho(x,y)=\frac{\d}{\d z}\frac{\phi(xz,yz)}{z}\Big{|}_{z=0}.
\label{vec}
\end{eqnarray}
Vector field is necessarily a pair of $2$-homogenic functions. For smooth functions, the functional equation (\ref{funkk}) implies the PDE
\begin{eqnarray}
u_{x}(x,y)(\varpi(x,y)-x)+
u_{y}(x,y)(\varrho(x,y)-y)=-u(x,y),\label{g-parteq}
\end{eqnarray}
and the same PDE for $v$, with boundary conditions
\begin{eqnarray*}
\lim\limits_{z\rightarrow 0}\frac{u(xz,yz)}{z}=x,\quad
\lim\limits_{z\rightarrow 0}\frac{v(xz,yz)}{z}=y.
\end{eqnarray*}
In principle, the PDE (\ref{g-parteq}) with the above boundary conditions is equivalent to (\ref{funkk}); but the proof of this equivalence is possible only in each case separately, since there arises the complication to determine the definition domain of the flow - for example, explore $\phi^{{\rm log}}$ below. Thus, not all functions in two complex variables can be iterated without restrictions - this restriction is explicitly present in (\ref{funkk}), but only implicitly present in (\ref{g-parteq}), and only \emph{a posteriori}, after we have found the solution, whether in analytic form in terms of known special functions, or via an independent analysis.\\

Each point under a flow possesses the orbit, which is either a curve (when we deal with flows over $\mathbb{R}$) or a surface (complex curve over $\mathbb{C}$), or a single point; the orbit is defined by
\begin{eqnarray*}
\mathscr{V}(\m{x})=\Big{\{}\frac{\phi(\m{x}z)}{z}:z\in\mathbb{R}\text{ or }\mathbb{C}\Big{\}}.
\end{eqnarray*}
For rational flows of level $N$, there exists an $N$th degree homogenic rational function $\mathscr{W}(x,y)$ such that the orbits are the curves $\mathscr{W}(x,y)=c$, $c\in\mathbb{C}\cup\{\infty\}$. We will also refer to non-rational flows as also being of level $N$ if the orbits are given by $N$th degree homogeneous rational function.
It appears that the vector field of a rational flow is a pair of $2$-homogenic rational functions. The main theorem in \cite{alkauskas2} gives the first structural result for the rational flows; many more properties of the flows which take into account not only the level but the birational conjugation itself will be treated in \cite{alkauskas3}. The main idea of the proof (which is rather lengthy) is that using conjugation with $1$-BIR, the vector field can be reduced step by step, leading to rational $2$-homogenic functions with numerators of smaller degree, and eventually, provided we do not encounter an obstruction, to quadratic forms. The obstruction arises when we hit the vector field whose both coordinates are proportional. So, we are left to find all rational flows whose vector field is given by a pair of two quadratic forms in two variables, or whose one of the coordinates vanish (this is equivalent, after a linear conjugation, to the obstruction). The first case is a rewarding part of the proof. For example, we find that some of these pairs of quadratic forms lead to rational solutions, while others should be discarded since they arise from non-rational flows. In this paper we take a closer look at non-rational solutions of (\ref{funkk}) whose vector field is a pair of quadratic forms. In general, every vector field whose orbit is given by a homogeneous rational function of degree $N\in\mathbb{N}_{0}$ gives rise, generally, not to a rational flow but to a quasi-rational flow; see the Subsection \ref{sub-qf} and \cite{alkauskas4}.
\subsection{Addition formulas. Motivation}Consider the following two-variable two-dimensional functions:
\begin{alignat*}{6}
\phi^{{\rm exp}}(\m{x})&= xe^{y}\bl y&&=xf(y)\bl y;\\
\phi^{{\rm tan}}(\m{x})&=\frac{xy+y^2\tan y}{y-x\tan y}\bl y&&=\frac{xy+y^2f(y)}{y-xf(y)}\bl y;\\
\phi_{ N}(\m{x})&=x(y+1)^{N-1}\bl\frac{y}{y+1}&&=xf(y)\bl\frac{y}{y+1};\\
\phi^{{\rm log}}(\m{x})&=\frac{xy}{(y+1)[y+x\log(y+1)]}\bl \frac{y}{y+1}&&=
\frac{xy}{(y+1)[y+xf(y)]}\bl \frac{y}{y+1}.
\end{alignat*}
(Here in each case $f(y)$ stands for the corresponding function, $e^{y}$, $\tan(y)$, and so on). We can check directly that these functions satisfy the PrTE (\ref{funkk}). The vector fields and the equations for the orbits of these flows are given by
\begin{center}
\begin{tabular}{r c l}
${\rm exp}$: & $xy\bl 0$, & $\mathscr{W}(x,y)=y$;\\
${\rm tan}$: & $x^2+y^2\bl 0$, & $\mathscr{W}(x,y)=y$;\\
        $N$: & $(N-1)xy\bl(-y^2)$, & $\mathscr{W}(x,y)=xy^{N-1}$;\\
${\rm log}$: & $-x^2-xy\bl(-y^2)$, & $\mathscr{W}(x,y)=\exp(yx^{-1})y$.
\end{tabular}
\end{center}
(In the last case the orbits are non-algebraic curves). In fact, replace the known function in the expression of $\phi$ with the unknown $f$ (this is shown above). Now require that the so obtained function satisfies (\ref{funkk}). Eventually, this turns out to be equivalent to nothing else but the standard addition formulas:
\begin{eqnarray*}
{\rm exp}&:&f(x+y)=f(x)f(y);\\
{\rm tan}&:&f(x+y)=\frac{f(x)+f(y)}{1-f(x)f(y)};\\
        N&:&f(xy-1)=f(x-1)f(y-1);\\
{\rm log}&:&f(xy-1)=f(x-1)+f(y-1).
\end{eqnarray*} The first two can be restated in a symmetric way; that is, if $u+v+w=0$, then
\begin{eqnarray*}
\exp u\cdot\exp v\cdot\exp w&=&1,\\
\tan u\cdot\tan v\cdot\tan w&=&\tan u+\tan v+\tan w.
\end{eqnarray*}
These are superior to the former, since if we know additionally that $\exp(0)=1$ and $\tan(0)=0$, then these identities also encode the symmetry properties $\exp(u)\exp(-u)=1$ and $\tan(u)+\tan(-u)=0$. 
In \cite{alkauskas2} we found out that the $k$-dimensional PrTE is directly related to birational transformations of $P^{k-1}(\mathbb{C})$, as opposed to general ``affine" translation equation which is tied to birational affine transformations of $\mathbb{C}^{k}$, albeit this dependency is of different nature, the first case being much more involved. Now we see another fascinating feature of (\ref{funkk}); namely, the PrTE provides a uniform framework for addition formulas for certain functions: abelian - $\exp(y)$, $\tan(y)$; algebraic - $(y+1)^{N-1}$, if $N\in\mathbb{Q}$; integrals over algebraic - $\log(y+1)$. This is the simplest, $2$-dimensional case of the theory. As we will see in Theorem \ref{thm1}, the above list will not include trigonometric functions $\sin(y)$ and $\cos(y)$, but will include rather special elliptic functions related to regular hexagonal and square lattices. The higher dimensional case even in the setting of \cite{alkauskas2} (that is, classification of higher dimensional rational projective flows) is open and very promising; see the Subsection \ref{dim}. \\

Now we formulate the main problem of this paper.
\begin{prob}Find all flows, that is, bivariate functions satisfying (\ref{funkk}), which are defined for $\m{x}\in\mathbb{C}^{2}\setminus\{\text{union of isolated curves}\}$, whose vector field is rational, and which are single-valued functions, i.e. without branching points. 
\label{main-prob}
\end{prob}

Thus, the flow $\phi^{{\rm log}}$ has an infinite branching point at $y=-1$. The flow $\phi_{N}$ has an infinite branching if $N\notin\mathbb{Q}$ and a finite one if $N\in\mathbb{Q}\setminus\mathbb{Z}$. So, our chief interest is only the case $N\in\mathbb{Z}$, and this was dealt with in \cite{alkauskas2}. On the other hand, the flows $\phi^{{\rm exp}}$ and $\phi^{{\rm tan}}$ are single-valued functions for all $\m{x}\in\mathbb{C}^{2}$. With the help of linear conjugation we can force the exponential and tangent flows to be symmetric with respect to the linear involution $(x,y)\mapsto(y,x)$. Thus, we summarize this as
\begin{prop}Let
\begin{eqnarray*}
\phi^{{\rm e}}=u^{{\rm e}}(x,y)&\bl& u^{{\rm e}}(y,x),\quad u^{e}(x,y)=\frac{1}{2}\big{(}(x-y)e^{x+y}+x+y\big{)},\text{ and}\\
\phi^{{\rm t}}=u^{{\rm t}}(x,y)&\bl& u^{{\rm t}}(y,x),
\quad u^{{\rm t}}(x,y)=\frac{(x+y)x+(x+y)y\tan(x+y)}{x+y+(y-x)\tan(x+y)}.
\end{eqnarray*}
These functions are symmetric projective flows. Both $u^{{\rm e}}$ and $u^{{\rm t}}$ satisfy the PDE (\ref{g-parteq}), where $\varpi(x,y)=\frac{1}{2}(x^2-y^2)=\varrho(y,x)$ in the exponential case, and $\varpi(x,y)=x^2+y^2=\varrho(y,x)$ in the tangential case, respectively. The projective flow property is equivalent to addition formulas for corresponding functions.
\label{prop1}
\end{prop}
Note however that these two cases are conjugate and will fall under the same item in the classification (see Theorem \ref{thm1}). In fact, let us define the $1$-BIR by
\begin{eqnarray*}
\ell(x,y)=\frac{x(x+y)}{y}\bl(x+y).
\end{eqnarray*}
Then
\begin{eqnarray*}
\ell^{-1}\circ\phi^{{\rm exp}}\circ\ell(x,y)=\frac{x(x+y)e^{x+y}}{x e^{x+y}+y}\bl\frac{y(x+y)}{x e^{x+y}+y},
\end{eqnarray*}
and the latter projective flow is linearly conjugate (over $\mathbb{C}$) to $\phi^{{\rm tan}}$. 
\subsection{Quadratic forms as vector fields of projective flows}
\label{sub-quadratic}
We will see later that along with $(y+1)^{N-1}$, $\tan(y)$ and $\exp(y)$, there are three more pairs of functions which complete the picture. As a crucial part of our investigations, let us make a brief summary of the Step II of the proof of the main theorem in \cite{alkauskas2}. Let the vector field of the projective flow $\phi=u\bl v$ be given by $\varpi(x,y)\bl \varrho(x,y)$, where both coordinates are $2$-homogenic rational functions, and the common denominator has a degree $d\geq 1$. We found that, unless $\varpi$ and $\varrho$ are proportional, there exists a $1$-BIR $\ell$ such that the vector field of the flow $\ell^{-1}\circ\phi\circ\ell$ is a pair of rational functions with lowered degree in the common denominator. Thus, we are left to consider cases where either $\varpi$ and $\varrho$ are proportional, or they both are quadratic forms. If the first statement holds, then we see that (after a linear conjugation) one can confine to the case $\varrho(x,y)\equiv 0$. This implies $v(x,y)=y$. If $\phi$ were a rational function, then this would necessarily mean that $u$ is a Jonqui\`{e}res transformation; this would imply that $\varpi$ is in fact a quadratic form, and calculations in (\cite{alkauskas2}, Step II) provide the complete solution. The flows $\phi^{{\rm exp}}$ and $\phi^{{\rm tan}}$ arise exactly from this analysis in cases $\varpi(x,y)=xy$ and $\varpi(x,y)=x^2+y^2$, respectively. In the setting of the Problem \ref{main-prob}, the pair $u(x,y)\bl y$ with the vector field $\varpi(x,y)\bl 0$ ought no longer be rational, the implication that $u$ is a Jonqui\`{e}res transformation is irrelevant, and we need to provide an independent analysis of the solution of (\ref{g-parteq}) in case $\varpi$ is any $2$-homogenic rational function, and $\varrho\equiv 0$. This is accomplished in the Subsection \ref{sub-zero}.\\

Suppose now that both $\varpi$ and $\varrho$ are quadratic forms, see (\cite{alkauskas2}, Step III). We found that if $x\varrho-y\varpi$ is a cube of a linear polynomial, then this always leads to the flow with a ramification of the type $\log(y+1)$. If $x\varrho-y\varpi$ is not a cube, then the vector field, with the help of linear conjugation, can be transformed into
\begin{eqnarray*}
\varpi(x,y)=ax^{2}+bxy,\quad\varrho(x,y)=cxy+dy^{2}.
\end{eqnarray*}
The cases $b=0$ or $c=0$ lead to rational solutions or flows with ramification of the type $(y+1)^{r}$, $r\notin\mathbb{Z}$. So, let $b,c\neq 0$. If $d=0$ or $a=0$, this again leads to a ramification of the type $\log(y+1)$. So, after a linear conjugation, we may assume that
\begin{eqnarray*}
\varpi(x,y)=x^{2}+Bxy,\quad\varrho(x,y)=Cxy+y^{2}.
\end{eqnarray*}
The flow with this vector field has ramifications of both types $(1-y)^{-B}$ and $(1-x)^{-C}$, so, if there is none, we get the arithmetic condition
\begin{eqnarray}
B,C\in\mathbb{Z}.
\label{arithm1}
\end{eqnarray}
If $B=1$, then again, $C=1$, otherwise the flow is ramified, and $B=C=1$ gives exactly the rational flow of level $0$. Assume $B,C\neq 0,1$. Then the vector field $\varpi\bl\varrho$ with the help of linear conjugation can be transformed into the vector field
\begin{eqnarray}
x^2+\frac{B+C-2}{BC-1}xy\bl Cxy+y^2.
\label{auto}
\end{eqnarray}
 Since this is also unramified, we get another arithmetic condition
\begin{eqnarray}
\frac{B+C-2}{BC-1}\in\mathbb{Z}.
\label{arithm2}
\end{eqnarray}
The case $B+C=2$ leads to a rational or algebraically ramified flow. It is fascinating that if $B,C\neq 0,1$, and $B+C\neq 2$, then all pairs $(B,C)$ which satisfy as simple arithmetic conditions as (\ref{arithm1}) and (\ref{arithm2}) are encoding elliptic unramified flows! More precisely, there are exactly $10$ such pairs (standard exercise):
\begin{eqnarray*}
& &(-2,-1)\leftrightarrow(-5,-1),\quad (-1,-2)\leftrightarrow(-5,-2),\quad(-2,-5)\leftrightarrow(-1,-5);\\
& &(-1,-3)\leftrightarrow(-3,-3),\quad(-3,-1)\circlearrowleft;\\
& &(-2,-2)\circlearrowleft.
\end{eqnarray*}
The symbol $``\leftrightarrow"$ means that the two pairs a linearly conjugate via (\ref{auto}), and $``\circlearrowleft"$ means that the flow is self conjugate. The pairs $(B,C)$ and $(C,B)$ are also linearly conjugate with the help of the involution $i(x,y)=(y,x)$. So, there are three equivalence classes of flows (shown as rows above), consisting of $6$, $3$ and $1$ pairs respectively; in each class any two flows are linearly conjugate, and we will show that all three arise form elliptic flows. We will choose such representatives: $(B,C)=(-2,-2)$, $(B,C)=(-3,-3)$, and $(B,C)=(-1,-2)$. Most of this paper deals with the first case, the fascinating vector field
\begin{eqnarray}
\varpi(x,y)=x^2-2xy,\quad\varrho(x,y)=y^2-2xy,
\label{specc}
\end{eqnarray}
which is the class on its own and thus it has exactly the $6$-fold symmetry: for every $\gamma\in\Sigma$,
\begin{eqnarray}
\gamma^{-1}\circ(\varpi,\varrho)\circ\gamma=(\varpi,\varrho);
\end{eqnarray}
for the definitions, see the property (SYMM), Subsection \ref{sub:main}. In the setting of \cite{alkauskas2}, all the tricks which ruled out other vector fields as arising from non-rational flows (as a rule, these tricks constituted in showing that corresponding flows have branching points, and rational flows, obviously, cannot have these) were not applicable in all these exceptional cases of $(B,C)$, and it was still not clear why the solution of (\ref{g-parteq}) in case (\ref{specc}), for example, which is exactly the function $\lambda(x,y)$ (see the Subsection \ref{sub:main}), cannot be a rational function. And it appears that it is not; since, if we put $f(z)=\lambda(z,-z)/z$, then, as the property (ELL), Subsection \ref{sub:main}, implies, one has
\begin{eqnarray*}
f(z)f(-z)[f(z)+f(-z)]\equiv2,
\end{eqnarray*}
and thus $(f(z),f(-z))$ parametrizes the elliptic curve $XY(X+Y)=2$ and $f(z)$ cannot be a rational function. In this case $\lambda(x,y)$ is not rational itself. Two other vector fields with $(B,C)=(-3,-3)$ and $(B,C)=(-1,-2)$, whose orbits are also elliptic curves, are examined in the Subsection \ref{two-other}.

\section{Auxiliary functions }Now we make an interlude and introduce functions which will be crucial in our study of the vector field (\ref{specc}). The material is just a collection of various facts from the literature; our contribution to this topic is an introduction of special elliptic functions $\s(u)$ and $\c(u)$ for which we prove Proposition \ref{basic-sc}. 
\subsection{The special hypergeometric series}
\label{hyper}
Let us introduce our first auxiliary function \cite{erdelyi}
\begin{eqnarray}
W(x)=\frac{1}{3}\int\limits_{0}^{1}\frac{\d t}{[(1-t)(1-xt)]^{2/3}}
=\,_{2}F_{1}\Big{(}\frac{2}{3},1;\frac{4}{3};x\Big{)},\quad -\infty<x<1.
\label{int}
\end{eqnarray}
It satisfies $W(0)=1$, and the linear ODE
\begin{eqnarray}
3x(1-x)W'(x)+(1-2x)W(x)=1.\label{prinz}
\end{eqnarray}
The derivative of this gives the second order ODE
\begin{eqnarray*}
3x(1-x)W''(x)+(4-8x)W'(x)-2W(x)=0,
\end{eqnarray*}
which coincided with Euler's hypergeometric differential equation for $(a,b;c)=(\frac{2}{3},1;\frac{4}{3})$. Since $c=2a$, the function $W(x)$ is invariant under Pfaff's transformation and thus it satisfies the functional equation \cite{erdelyi}
\begin{eqnarray*}
W(x)=\frac{1}{1-x}W\Big{(}\frac{x}{x-1}\Big{)},\quad x<0;
\end{eqnarray*}
this can be verified easily using the integral representation (\ref{int}). Of course, this functional equation is satisfied by all hypergeometric functions of the form $\,_{2}F_{1}(a,1;2a;x)$. For example, when $a=1$, this hypergeometric function reduces to $-x^{-1}\log(1-x)$. The change of variables $1-xt=(1-x)T$ in (\ref{int}) gives
\begin{eqnarray}
W(x)=\frac{1}{3[x(1-x)]^{1/3}}\int\limits_{1}^{\frac{1}{1-x}}\frac{\d t}{[t(t-1)]^{2/3}},\quad x<1.\label{everyell}
\end{eqnarray}
The appearance of the symmetric group $S_{3}$ in our investigations - see (SYMM), the Subsection \ref{sub:main} - can be explained from several points of view; here is one of them.\\
 
 The special case of Kummer's theory for hypergeometric series \cite{erdelyi} is the following fact, which is easily checked in our case: the differential equation (\ref{prinz}) has the following three solutions:
 \begin{center}
\begin{tabular}{l  c  l}
$W_{0}(x)=W(x)$& for $-\infty<x<1$,\quad & $W_{0}(0)=1$;\\
$W_{1}(x)=-W(1-x)$& for $0<x<\infty$,\quad & $W_{1}(1)=-1$;\\
$W_{\infty}(x)=\displaystyle{\frac{1}{x}W\Big{(}\frac{1}{x}\Big{)}}$&
 for $x\in(-\infty,0)\cup(1,\infty)$,\quad & $W_{\infty}(\infty)=0$.\\
\end{tabular}
\end{center}
This can be treated as the complete description of the differential equation (\ref{prinz}) on the line $P^{1}(\mathbb{R})$. The three singular points $0,1,\infty$ divide this circle into three parts. For each interval $(a,b)$, $a,b\in\{0,1,\infty\}$, $a\neq b$, the general solution of (\ref{prinz}) in the interval $(a,b)$ with a floating boundary condition is given by
\begin{eqnarray*}
pW_{a}(x)+(1-p)W_{b}(x),\quad p\in\mathbb{R}.
\end{eqnarray*}
The symmetric group $S_{3}$ acts on the set $\{0,1,\infty\}$ by permutations. This corresponds to the action of M\"{o}bius transformations on $W_{0}$, $W_{1}$ and $W_{\infty}$ as follows.
\begin{eqnarray*}
f(x)&\mapsto&\frac{1}{x}f\Big{(}\frac{1}{x}\Big{)}\text{ corresponds to }(0\,\infty)(1),\\
f(x)&\mapsto& -f(1-x)\text{ corresponds to }(0\,1)(\infty),\\
f(x)&\mapsto& \frac{1}{1-x}f\Big{(}\frac{x}{x-1}\Big{)}\text{ corresponds to }(1\,\infty)(0).
\end{eqnarray*}
The first entry, for instance, means the following: the map under consideration interchanges $W_{0}$ and $W_{\infty}$ but leaves the function $W_{1}$ intact. Other two elements of $S_{3}$ (the last one is the identity) are obtained from the above. So, for example, the cycle $(0\,1\,\infty)=(0\,\infty)(1)\cdot(1\,\infty)(0)$ corresponds to the transformation
\begin{eqnarray*}
f(x)&\mapsto& -\frac{1}{x}f\Big{(}\frac{x-1}{x}\Big{)}.
\end{eqnarray*}
These correspondences should find their analogues in the setting of the Subsection \ref{dim}, at least in case of symmetric groups $S_{N}$, $N\geq 4$.
\subsection{The Dixonian elliptic functions}
\label{sub-dix}
For the general theory of elliptic functions we may refer to the classical book (in Russian) \cite{ahiezer}. Let $\omega=e^{2\pi i/3}$. The functions $\sm(u)$, $\cm(u)$ were introduced by Dixon \cite{dixon} as a pair of functions which parematrize the Fermat cubic $X^3+Y^3=1$. These are in fact special elliptic functions satisfying the following
\begin{prop}{\cite{dixon,flajolet-c}} The Dixonian elliptic functions have these properties
\begin{eqnarray*}
\sm(0)=0,\cm(0)=1,&\quad&\sm^{3}(u)+\cm^{3}(u)\equiv 1,\\
\sm'(u)=\cm^{2}(u),&\quad&\cm'(u)=-\sm^{2}(u),\\
\sm(-u)=-\frac{\sm(u)}{\cm(u)},&\quad&\cm(-u)=\frac{1}{\cm(u)},\\
\sm(\omega u)=\omega\sm(u),\cm(\omega u)=\cm(u),&\quad&
\sm\Big{(}\frac{\pi_{3}}{3}-u\Big{)}=\cm(u),\\
{\rm sm}(u+v)&=&\frac{s_{1}c^{2}_{2}+s_{2}c^{2}_{1}-s^{2}_{1}s^{2}_{2}c_{1}c_{2}}{1-s^{3}_{1}s^{3}_{2}},\\
{\rm cm}(u+v)&=&\frac{c_{1}c_{2}-s_{1}s_{2}(s_{1}c^2_{2}+s_{2}c^2_{1})}{1-s^{3}_{1}s^{3}_{2}},\\
\text{here }s_{1}=\sm(u),c_{1}=\cm(u),& &s_{2}=\sm(v),c_{2}=\cm(v),
\end{eqnarray*}
\begin{eqnarray*}
\sm(u)&=&u-4\frac{u^4}{4!}+160\frac{u^7}{7!}-20800\frac{u^{10}}{10!}+647680\frac{u^3}{13!}+\cdots,\\
\cm(u)&=&1-2\frac{u^3}{3!}+40\frac{u^6}{6!}-3680\frac{u^{9}}{9!}
+880000\frac{u^{12}}{12!}+\cdots,\\
\cm\Big{(}x^{1/3}(x-1)^{1/3}W(x)\Big{)}&=&\frac{1}{(1-x)^{1/3}},\quad
\sm\Big{(}x^{1/3}(x-1)^{1/3}W(x)\Big{)}=\frac{(-x)^{1/3}}{(1-x)^{1/3}};
\end{eqnarray*}
\label{dixonian}
the last holds for $-\infty<x<1$.
Here the contant $\pi_{3}$ (the period, according to the notion of Zagier-Kontsevich) is given by
\begin{eqnarray*}
\pi_{3}=B\Big{(}\frac{1}{3},\frac{1}{3}\Big{)}=
\frac{\Gamma(\frac{1}{3})^{2}}{\Gamma(\frac{2}{3})}
=\frac{\sqrt{3}}{2\pi}\Gamma\Big{(}\frac{1}{3}\Big{)}^{3}=5.299916250856_{+},\text{ and }
\Pi=\frac{\pi^{3}_{3}}{27}=5.513701576710_{+}.
\end{eqnarray*}
\end{prop}
The last two constants are of great importance in the current paper. The can be numerically calculated by the fast converging series
\begin{eqnarray*}
\pi^{6}_{3}=8\pi^{6}\Bigg{(}1-\sum\limits_{n=1}^{\infty}
\frac{504n^5}{(-1)^ne^{\sqrt{3}\pi n}-1}\Bigg{)}=820.824437079556_{+}.
\end{eqnarray*}
The full lattice of periods for both $\sm(u)$ and $\cm(u)$ is given by $\mathbb{Z}\pi_{3}\oplus\mathbb{Z}\pi_{3}\omega$. Let
$\mathcal{F}=\{\pi_{3}s+\pi_{3}\omega t:s,t\in[0,1)\}$ be the fundamental parallelogram. Consider $9$ special points
\begin{eqnarray*}
q_{a,b}=\frac{a\pi_{3}}{3}+\frac{b\pi_{3}\omega}{3},\quad 0\leq a,b\leq 2.
\end{eqnarray*}
In terms of the theory of elliptic functions, both $\sm(u)$ and $\cm(u)$ are order $3$ elliptic functions, and so each of them attains every value in $\mathbb{C}\cup\{\infty\}$ in $\mathcal{F}$ exactly thrice, counting multiplicities. The simple poles of both $\sm(u)$ and $\cm(u)$ are $q_{2,0}$, $q_{1,1}$ and $q_{0,2}$. The function $\sm(u)$ has simple zeros at $q_{0,0}$, $q_{2,1}$ and $q_{1,2}$, while the values at $q_{1,0}$, $q_{0,1}$ and $q_{2,2}$ are, respectively, $1,\omega,\omega^2$.
Likewise, $\cm(u)$ has simple zeros at $q_{1,0}$, $q_{0,1}$ and $q_{2,2}$, and the values at $q_{0,0}$, $q_{1,2}$ and $q_{2,1}$ are, respectively, $1$ (triple value), $\omega$, $\omega^2$. These properties follow from Proposition \ref{dixonian}. The authors in \cite{flajolet-c}, for the convenience reasons, introduce the hyperbolic versions of these functions, given by ${\rm smh}(u)=-{\rm sm}(-u)$, ${\rm cmh}(u)={\rm cm}(-u)$. These functions parametrize the ``Fermat hyperbola" $y^{3}-x^{3}=1$. For our purposes, we will need yet another version of these functions, and here we introduce
\begin{eqnarray*}
\s(u)=-\frac{{\rm sm}^{2}(u)}{{\rm cm}(u)},\quad
\c(u)=\frac{{\rm cm}^{2}(u)}{{\rm sm}(u)}.
\end{eqnarray*}
These are order $6$ elliptic function with the same period lattice $\mathbb{Z}\pi_{3}\oplus\mathbb{Z}\pi_{3}\omega$. Moreover, we have
\begin{prop}The functions $\s(u)$ and $\c(u)$ have these properties:
\begin{eqnarray*}
1&\equiv&\s(u)\c(u)[\s(u)-\c(u)],\\
\s'(u)&=&-\s^{2}(u)+2\s(u)\c(u),\\
\c'(u)&=&-\c^{2}(u)+2\s(u)\c(u),\\
\s(\omega u)&=&\omega^{2}\s(u),\c(\omega u)=\omega^{2}\c(u),\\
\s(-z)&=&\s(z),\c(-z)=\frac{1}{\s(z)\c(z)},\\
\c\Big{(}\frac{\pi_{3}}{3}-u\Big{)}&=&-\s(u),\\
\s(u+v)&=&\frac{(C_{1}+C_{2}-S_{1}S_{2}C_{1}C_{2})^{2}S_{1}S_{2}}
{(1-S_{1}^{2}S_{2}^{2}C_{1}C_{2})(S_{1}S_{2}C_{1}+S_{1}S_{2}C_{2}-1)},\\
\c(u+v)&=&\frac{(1-S_{1}S_{2}C_{1}-S_{1}S_{2}C_{2})^2C_{1}C_{2}}
{(1-S_{1}^{2}S_{2}^{2}C_{1}C_{2})(C_{1}+C_{2}-S_{1}S_{2}C_{1}C_{2})};\\
\text{here }S_{1}&=&\s(u),\quad C_{1}=\c(u),\\
            S_{2}&=&\s(v),\quad C_{2}=\c(v).
\end{eqnarray*}
\label{basic-sc}
\end{prop}
These are verified using Proposition \ref{dixonian}.

\section{The results}
\subsection{Classification}The first main result of this paper is the complete solution of the Problem \ref{main-prob}.
\begin{thm}Let $\phi(\m{x})=u(x,y)\bl v(x,y)$ be a smooth projective flow such that (\ref{bound}) holds, and its vector field, defined by (\ref{vec}), is a pair of $2$-homogenic rational functions. Suppose that both $u$ and $v$ are defined on $\mathbb{C}^{2}$, except for a countable set of isolated curves each, where they might have poles, and that $u,v$ in their definition domains are single-valued analytic functions. Then there exists a $1$-BIR $\ell$, such that $\ell^{-1}\circ\phi\circ\ell(\m{x})$ is one of the following canonic projective flows:
\begin{itemize}
\item[1)] $x\bl y$;
\item[2)] $\phi_{N}$ for $N\in\mathbb{N}_{0}$, the level $N$ flow, whose orbits are given by $xy^{N-1}=c$; only in the latter two cases the flow is rational;
\item[3)] $\phi^{{\rm e}}(\m{x})$, the level $1$ flow whose orbits are given by $x+y=c$; it is conjugate to the flow $\phi^{{\rm t}}(\m{x})$, also the level $1$ flow and also with orbits given by $x+y=c$ (for these two, see Proposition \ref{prop1});
\item[4)] $\Lambda(\m{x})=\lambda(x,y)\bl \lambda(y,x)$, the level $3$ flow whose vector field is $x^2-2xy\bl y^2-2xy$, orbits are given by $xy(x-y)=c$, and this flow is algebraically expressable in terms of Dixonian elliptic functions (see Theorem \ref{thm2});
\item[5)] $\Psi(\m{x})=\psi(x,y)\bl\psi(y,x)$, the level $4$ flow whose vector field is $x^2-3xy\bl y^2-3xy$, orbits are given by $xy(x-y)^2=c$, and this flow is expressable in terms of lemniscatic elliptic functions (with quadratic period lattice);
\item[6)] $\Delta(\m{x})=\alpha(x,y)\bl\beta(x,y)$, the level $6$ flow whose vector field is $x^2-xy\bl y^2-2xy$, orbits are given by $(3x-2y)x^3y^2=c$, and this flow is expressable in terms of Dixonian elliptic functions  again (for the last two, see the Subsection \ref{two-other}).
\end{itemize}
\label{thm1}
\end{thm}
Our main concern of this paper is the case 4); the last two cases will be covered in the Subsection \ref{two-other}. The case 4) is particularly interesting since it has an additional $6$-fold symmetry, and now we will concentrate on it.
\subsection{The $S_{3}$-superflow} (For the explanation of the title, see the Subsection \ref{dim}).
\label{sub:main}
So, we investigate the fascinating function $\Lambda(x,y)=\lambda(x,y)\bl\lambda(y,x)$ which possesses these main properties.
\begin{itemize}
\item{(PDE). }The first coordinate satisfies the partial differential equation
\begin{eqnarray}
\lambda_{x}(x,y)(x^2-2xy-x)+\lambda_{y}(x,y)(y^2-2xy-y)=-\lambda(x,y)\label{parteq}
\end{eqnarray}
with the boundary condition
\begin{eqnarray}
\lim\limits_{z\rightarrow 0}\frac{\lambda(xz,yz)}{z}=x.\label{boundd}
\end{eqnarray}
\item{(FLOW). }The function $\Lambda$ satisfies the iterative functional equation
\begin{eqnarray*}
(1-z)\Lambda(\m{x})=\Lambda\Big{(}\Lambda(\m{x}z)\frac{1-z}{z}\Big{)},\quad\m{x}=(x,y)\in\mathbb{C}^{2},z\in\mathbb{C}.\label{funk}
\end{eqnarray*}
\item{(SYMM). }If $\gamma=\bigl( \begin{smallmatrix} a&b\\ c&d \end{smallmatrix} \bigr)$, let $\gamma(x,y)=(ax+by,cx+dy)$. Consider the $6$ element subgroup of ${\rm GL}_{2}(\mathbb{R})$, call it $\Sigma\cong S_{3}$, whose elements are
\begin{eqnarray*}
i=\left(\begin{array}{cc}1 & 0 \\0 & 1 \\ \end{array}\right),\quad
\sigma=\left(\begin{array}{cc}0 & 1 \\1 & 0 \\ \end{array}\right),\quad
\tau=\left(\begin{array}{cc}1 & -1 \\0 & -1 \\ \end{array}\right),
\end{eqnarray*}
and
\begin{eqnarray*}
\sigma\tau\sigma=\left(\begin{array}{cc}-1 & 0 \\-1 & 1 \\ \end{array}\right),\quad
\sigma\tau=\left(\begin{array}{cc}0 & -1 \\1 & -1 \\ \end{array}\right),\quad
\tau\sigma=\left(\begin{array}{cc}-1 & 1 \\-1 & 0 \\ \end{array}\right);
\end{eqnarray*}
then $(\sigma\tau)^{3}=(\tau\sigma)^{3}=i$, $\sigma^2=\tau^2=(\sigma\tau\sigma)^2=i$.
The function $\Lambda(x,y)$ possesses the $6$-fold symmetry: for every $\gamma\in\Sigma$, we have
\begin{eqnarray*}
\gamma^{-1}\circ\Lambda\circ\gamma(x,y)=\Lambda(x,y).
\end{eqnarray*}
Of course, involutions $\sigma$ and $\tau$ generate the whole group $\Sigma$, so only two of these invariance properties are independent. The invariance under $\sigma$ tells us that $\Lambda(x,y)=(\lambda(x,y),\lambda(y,x))$, that is, the second coordinate is just the flip of the first, and the invariance under $\tau$ implies the identities
\begin{eqnarray}
\left\{\begin{array}{c@{\qquad}l}
\lambda(x,y)+\lambda(-y,x-y)+\lambda(y-x,-x)=0,\\
\lambda(x,y)+\lambda(-x,y-x)=0.
\label{period}\end{array}\right.
\end{eqnarray}
\item{(ELL). }In case $x,y$ are fixed, $xy(x-y)\neq 0$, the pair of functions
\begin{eqnarray*}
(X(z),Y(z))=\Big{(}\frac{\lambda(xz,yz)}{z},
\frac{\lambda(yz,xz)}{z}\Big{)}=(\lambda^{z}(x,y),\lambda^{z}(y,x))
\end{eqnarray*}
parametrize the elliptic curve $XY(X-Y)=xy(x-y)$;
thus,
\begin{eqnarray*}
\frac{\lambda(xz,yz)}{z}\cdot\frac{\lambda(yz,xz)}{z}\cdot\Big{(}\frac{\lambda(xz,yz)}{z}-\frac{\lambda(yz,xz)}{z}\Big{)}\equiv xy(x-y),\quad x,y,z\in\mathbb{C}.
\end{eqnarray*}
It turns out that three exceptional lines $y=0,x-y=0$, and $x=0$ correspond to three ramification points of the algebraic function $[z(1-z)]^{-2/3}$; namely, $z=0,1$, and $\infty$.
\end{itemize}
\begin{figure}
\epsfig{file=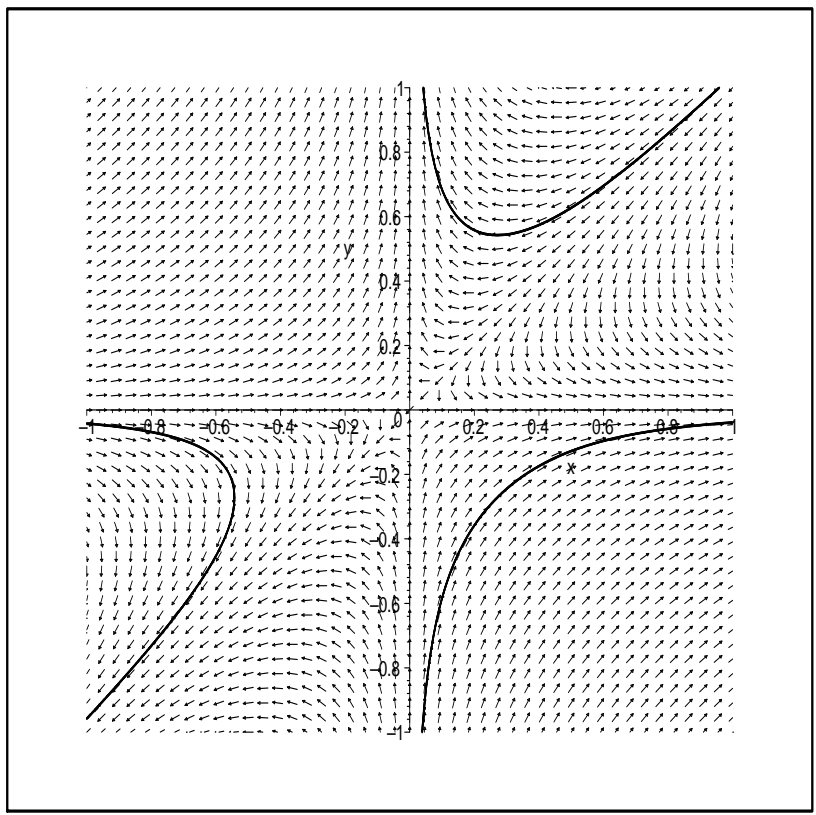,width=305pt,height=305pt,angle=-90}
\caption{The vector field of the flow $\Lambda(\m{x})$ is given by $x^2-2xy\bl y^2-2xy$. Here we show the normalized vector field, meaning that all arrows are adjusted to have the same length. A selected orbit shown is the elliptic curve $xy(x-y)=-0.04.$ A point travels this orbit with a convention that, for example, if it goes up via the left side of the top-right branch, it reappears on the left side of the bottom-right branch, and so on. So, we always mind the asymptote.}
\end{figure}
\subsection{Basic properties of $\lambda(x,y)$}
The PDE (\ref{parteq}) with the boundary condition (\ref{boundd}) has, as already mentioned, the unique solution $\lambda(x,y)$. In this subsection we will derive few computational results. As was proved in \cite{alkauskas2}, we formally have 
\begin{eqnarray}
\lambda(xz,yz)=xz+\sum\limits_{i=2}^{\infty}z^{i}\varpi^{(i)}(x,y),
\label{seru}
\end{eqnarray}
and the homogenic functions $\varpi^{(i)}(x,y)$ can be recurrently calculated by
\begin{eqnarray}
\varpi^{(i+1)}(x,y)=\frac{1}{i}\,[\varpi^{(i)}_{x}(x,y)\varpi(x,y)+\varpi^{(i)}_{y}(x,y)\varrho(x,y)],\quad i\geq 2;
\label{upow}
\end{eqnarray}
here, as before, $\varpi(x,y)=x^2-2xy$, $\varrho(x,y)=y^2-2xy$. This recursion is essentially equivalent to the PDE (\ref{parteq}). The above also holds for $i=1$ if we make a natural convention that $\varpi^{(1)}(x,y)=x$. Thus, we can calculate these polynomials, as presented in the Table \ref{tablo}. \\
\small
 \begin{table}[h]
\begin{tabular}{c | l}
$i$ & $\varpi^{(i)}(t,1)=\w_{i}(t)$. \\
\hline\hline
$1$ & $t$\\
$2$ & $t^2-2t$\\
$ $ & $ $ \\
$3$ & $t^3-t^2+t$\\
$ $ & $ $ \\
$4$ & $t^4-2t^3$\\
$ $ & $ $ \\
$5$ & $t^5-\frac{5}{2}t^4+\frac{5}{2}t^3$\\
$ $ & $ $ \\
$6$ & $t^6-3t^5+3t^4-2t^3$\\
$ $ & $ $ \\
$7$ & $t^7-\frac{7}{2}t^6+\frac{9}{2}t^5-2t^4+t^3$\\
$ $ & $ $ \\
$8$ & $ t^8-4t^7+\frac{43}{7}t^6-\frac{32}{7}t^5+\frac{5}{7}t^4-\frac{2}{7}t^3$\\
$ $ & $ $ \\
$9$ & $t^9-\frac{9}{2}t^8+\frac{225}{28}t^7-\frac{199}{28}t^6+\frac{51}{14}t^5-\frac{3}{28}t^4+\frac{1}{28}t^3$\\
$ $ & $ $ \\
$10$ & $t^{10}-5t^9+\frac{285}{28}t^8-\frac{75}{7}t^7+\frac{165}{28}t^6-\frac{33}{14}t^5$\\
$ $ & $ $ \\
$11$ & $t^{11}-\frac{11}{2}t^{10}+\frac{88}{7}t^9-\frac{29}{28}t^8+\frac{297}{28}t^7-\frac{99}{28}t^6+\frac{33}{28}t^5$\\
$ $ & $ $ \\
$12$ & $t^{12}-6t^{11}+\frac{213}{14}t^{10}-\frac{295}{14}t^9+\frac{120}{7}t^8-\frac{117}{14}t^7+\frac{3}{2}t^6-\frac{3}{7}t^5$\\
$ $ & $ $ \\
$13$ & $t^{13}-\frac{13}{2}t^{12}+\frac{507}{28}t^{11}-\frac{1573}{56}t^{10}+\frac{1475}{56}t^9-\frac{843}{56}t^8+\frac{309}{56}t^7-\frac{3}{7}t^6+\frac{3}{28}t^5$\\
$ $ & $ $ \\
$14$ & $t^{14}-7t^{13}+\frac{85}{4}t^{12}-\frac{73}{2}t^{11}+\frac{3527}{91}t^{10}-\frac{4741}{182}t^9+\frac{138}{13}t^8-\frac{285}{91}t^7+\frac{27}{364}t^6-\frac{3}{182}t^5$\\
$ $ & $ $ \\
$15$ & $t^{15}-\frac{15}{2}t^{14}+\frac{345}{14}t^{13}-\frac{325}{7}t^{12}+\frac{140325}{2548}t^{11}-\frac{108123}{2548}t^{10}+\frac{53905}{2548}t^9
-\frac{1095}{182}t^8+\frac{3855}{2548}t^7-\frac{15}{2548}t^6+\frac{3}{2548}t^5$.\\
$ $ & $ $ \\
\hline
\end{tabular}
\caption{Polynomials $\mathfrak{w}_{i}(t)$.}
\label{tablo}
\end{table}
\normalsize

The two pairs of functions $(x,y)$ and $(x^2-2xy,y^2-2xy)$ are invariant under conjugation with two independent linear involutions $(x,y)\mapsto(y,x)$ and $(x,y)\mapsto(x-y,-y)$. 
We thus get that polynomials $\w_{n}(t)=\varpi^{(n)}(t,1)$ possess the same $6$-fold symmetry: 
\begin{eqnarray*}
\left\{\begin{array}{c@{\qquad}l}
\displaystyle{\w_{n}(t)+(-t)^{n}\w_{n}\Big{(}1-
\frac{1}{t}\Big{)}+(t-1)^{n}\w_{n}\Big{(}\frac{1}{1-t}\Big{)}=0},\\
\displaystyle{\w_{n}(t)+(1-t)^{n}\w_{n}\Big{(}\frac{t}{t-1}\Big{)}=0}.
\end{array}\right.
\end{eqnarray*}
Moreover,
\begin{eqnarray*}
\w_{n}(t)\equiv 0\text{ mod }t^{\jmath(n)},\quad
\jmath(n)=2\Big{\lfloor}\frac{n+2}{6}\Big{\rfloor},
\end{eqnarray*}
and this congruence is exact, meaning that the next higher power of $t$ is the smallest power present in the polynomial $\w_n(t)$. Computer calculations show that this fact uniquely characterizes $\w_{n}(t)$ among monic polynomials with $6$-fold symmetry only for small $n$. Let
\begin{eqnarray*}
\lambda(x,y)=\sum\limits_{n=0}^{\infty}x^{n}f_{n}(y),\quad f_{n}(y)=\frac{1}{n!}\frac{\partial^{n}}{\partial x^{n}}\lambda(x,y)\Big{|}_{x=0}\in\mathbb{C}[[y]].
\end{eqnarray*}
The functions $f_{n}(y)$ are in fact all polynomials:
\begin{eqnarray*}
f_{1}(y)&=&1-2y+y^2,\\
f_{2}(y)&=&1-y,\\
f_{3}(y)&=&1-2y+\frac{5}{2}y^2-2y^3+y^4-\frac{2}{7}y^5+\frac{1}{28}y^6,\\
f_{4}(y)&=&1-\frac{5}{2}y+3y^2-2y^3+\frac{5}{7}y^4-\frac{3}{28}y^5,\\
f_{5}(y)&=&1-3y+\frac{9}{2}y^2-\frac{32}{7}y^3+\frac{51}{14}y^4-\frac{33}{14}y^5+\frac{33}{28}y^6-\frac{3}{7}y^7+\frac{3}{28}y^8
-\frac{3}{182}y^9+\frac{3}{2548}y^{10},\\
f_{6}(y)&=&1-\frac{7}{2}y+\frac{43}{7}y^2-\frac{199}{28}y^3+\frac{165}{28}y^4-\frac{99}{28}y^5+\frac{3}{2}y^6-\frac{3}{7}y^7+\frac{27}{364}y^8-\frac{15}{2548}y^9,
\end{eqnarray*}
and so on. By the direct calculation, using (\ref{upow}), we get (computationally now, which is automatic \emph{a posteriori} we finish the proof of the Theorem \ref{thm2}) that
\begin{eqnarray}
\lambda(0,x)&=&0,\quad \lambda(x,x)=\frac{x}{1+x},\quad \lambda(x,0)=\frac{x}{1-x},\\
\lambda_{x}(x,x)&=&\frac{1}{2}(x+1)^{-2}+\frac{1}{2}(x+1)^{2},\quad
\lambda_{y}(x,x)=\frac{1}{2}(x+1)^{-2}-\frac{1}{2}(x+1)^{2},\nonumber
\label{sing}
\end{eqnarray}
but
\begin{eqnarray}
\begin{split}
\lambda(x,-x)
&=x+3x^2+3x^3+3x^4+6x^5+9x^6+12x^7+\frac{117}{7}x^8+\frac{171}{7}x^9\\
&+\frac{246}{7}x^{10}
+\frac{348}{7}x^{11}+\frac{495}{7}x^{12}+\frac{708}{7}x^{13}+\frac{13140}{91}x^{14}+\frac{131076}{637}x^{15}\\
&+\frac{186903}{637}x^{16}
+\frac{266670}{637}x^{17}+\frac{380403}{637}x^{18}+\frac{542532}{637}x^{19}+\frac{1130958}{931}x^{20}\\
&+\frac{20971530}{12103}x^{21}+\frac{209391300}{84721}x^{22}+\frac{298661544}{84721}x^{23}
+\frac{425993769}{84721}x^{24}+\cdots.\\
\end{split}\label{skew}
\end{eqnarray}
The orbits of the flow $\Lambda$ are given by $\mathscr{W}(x,y)=xy(x-y)=c\in\mathbb{C}$. Consider the plane cubic (in affine coordinates) $XY(X-Y)=c$.
When $c\neq 0$, this is a non-singular cubic and thus an elliptic curve. If
\begin{eqnarray*}
X=\frac{c-q}{2p},\quad Y=\frac{-c-q}{2p},\end{eqnarray*}
we get it in Weierstrass form (also in affine coordinates)
\begin{eqnarray*}
(p:q:1)\in P^{2}(\mathbb{C}):q^2=4p^3+c^2.
\end{eqnarray*}
\subsection{Hypergeometric approach}The function $\lambda(x,y)$ can be explored in two different (mutually inverse) ways - either using hypergeometric function $W(x)$, or using Dixonian elliptic functions $\sm(u)$ and $\cm(u)$. We choose the second way, but will now exhibit how the first approach does work. Both methods reduce in fact to algebraic manipulations in quotient rings of rings of rational functions.  Let us define the curve $\mathscr{C}_{0}\subset\mathbb{R}^{2}$ parametrically by (see Figure \ref{figure2})
\begin{eqnarray*}
\mathscr{C}_{0}=\Big{\{}\Big{(}xW(x),W(x)\Big{)}:-\infty<x<1\Big{\}}.
\end{eqnarray*}
\begin{prop}The function $\lambda(x,y)$ vanishes on the curve $\mathscr{C}_{0}$:
\begin{eqnarray*}
\lambda(xW(x),W(x))\equiv 0 \text{ for }-\infty<x<1.
\end{eqnarray*}
\end{prop}
\begin{proof} Let $P(x)=\lambda(xW(x),W(x))$. When $x=0$, we have $P(0)=0$. Let $A=xW(x)$, $B=W(x)$. Then
\begin{eqnarray}
P'(x)=\lambda_{x}(A,B)[W(x)+xW'(x)]+\lambda_{y}(A,B)W'(x).\label{zer}
\end{eqnarray}
Note that the differential equation (\ref{prinz}) for $W(x)$ implies
\begin{eqnarray*}
A^{2}-2AB-A&=&[W(x)+xW'(x)]\cdot W(x)3x(x-1),\\
B^{2}-2AB-B&=&W'(x)\cdot W(x)3x(x-1).
\end{eqnarray*}
Thus, let us multiply (\ref{zer}) by $W(x)3x(x-1)$. Then the PDE for $\lambda(x,y)$ implies
\begin{eqnarray*}
P'(x)\cdot\Big{(}W(x)3x(x-1)\Big{)}&=&
\lambda_{x}(A,B)(A^2-2AB-A)+\lambda_{y}(A,B)(B^2-2AB-B)\\
&=&
-\lambda(A,B)=-P(x),\text{ for }-\infty<x<1.
\end{eqnarray*}
This is the differential equation for $P(x)$, and, minding $P(0)=0$, its only solution is $P(x)\equiv 0$.
\end{proof}
The curve $\mathscr{C}_{0}$ has the asymptote $x=y$. More precisely, it is much closer to the cubic
\begin{eqnarray*}
\Big{(}x,x+\frac{\Pi}{x^2}\Big{)},\quad x>1.
\end{eqnarray*}
Let $E(c)=\{(x,y)\in\mathbb{C}^{2}:xy(x-y)=c\}$, and $E^{\mathbb{R}}(c)=\{(x,y)\in\mathbb{R}^{2}:xy(x-y)=c\}$. Not every orbit (elliptic curve) $E^{\mathbb{R}}(c)$ for $c\neq 0$ intersects the curve $\mathscr{C}_{0}$, but only those with $c\in(-\Pi,\Pi)\setminus\{0\}$. This follows from the inspection of the function $W^{3}(x)x(1-x)$, which is monotonically increasing from $-\Pi$ to $\Pi$ in the interval $(-\infty,1)$; see the representation (\ref{everyell}). The Figure \ref{figure2} shows the boundary case $E^{\mathbb{R}}(\Pi)$, where this elliptic curve and $\mathscr{C}_{0}$ touch at infinity. On the curve 
\begin{eqnarray*}
\mathscr{C}_{\infty}=\Big{\{}\Big{(}W(x),xW(x)\Big{)}:-\infty<x<1\Big{\}}
\end{eqnarray*}
the function $\lambda(x,y)$ attains the value $\infty$, and the same holds on the curve
\begin{eqnarray*}
\mathscr{C}_{1}=\Big{\{}\Big{(}(x-1)W(x),-W(x)\Big{)}:-\infty<x<1\Big{\}}.
\end{eqnarray*}
The curve $\mathscr{C}_{0}$ as a whole remains intact under the linear involution $(x,y)\mapsto(-x,y-x)$, while the other two curves interchange. This and similar properties follow from the Subsection \ref{hyper}. This shows how the action of the group $S_{3}$ on the function $W(x)$ can be interpreted as its action on $\lambda(x,y)$. Despite the possibility to develop all properties of $\lambda(x,y)$ over $\mathbb{C}^{2}$ (not just over $\mathbb{R}^{2}$) in the framework of $W(x)$, henceforth we choose the elliptic function setting.
\subsection{Analytic formulas}
\label{sub:ann}
 Our second main result of this paper reads as follows.
\begin{thm}The function $\lambda(x,y)$ can be given the analytic expression:
\begin{eqnarray*}
\lambda(x,y)&=&\frac{\v\big{(}c\v^2-sc^2y\v+s^2xy\big{)}^{2}}
{y\big{(}x-c^3y\big{)}\big{(}c^2\v^2-sx\v+s^2cxy\big{)}},\\
\lambda(y,x)&=&\frac{\v\big{(}c^2\v^2-sx\v+s^2cxy\big{)}^{2}}
{x\big{(}x-c^3y\big{)}\big{(}c\v^2-sc^2y\v+s^2xy\big{)}};
\end{eqnarray*}
here $\v=\v(x,y)=[xy(x-y)]^{1/3}$, and $s=\sm(\v),c=\cm(\v)$ are the Dixonian elliptic functions. This function satisfies all the above properties. In particular, $\lambda(x,y)$ is a single-valued $2$-variable analytic function.
\label{thm2}
\end{thm}
\begin{figure}
\epsfig{file=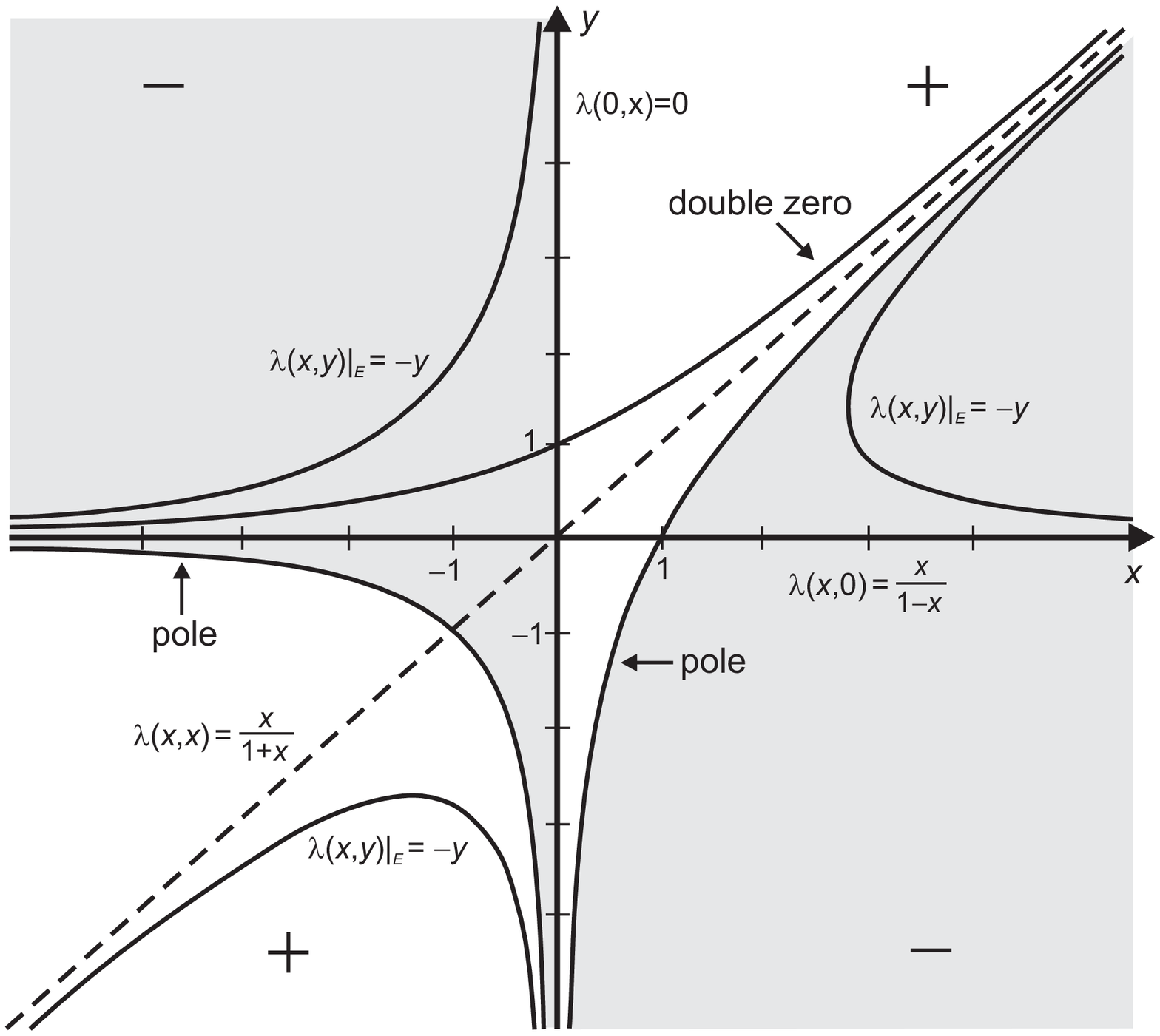,width=340pt,height=315pt,angle=0}
\caption{The function $\lambda(x,y)$ for $(x,y)\in\mathbb{R}^{2}$; it is a real section of the bivariate analytic function. The grey and white regions show where this function is negative and, respectively, positive. The double zero is the curve $\mathscr{C}_{0}$ parametrized by $(xW(x),W(x))$, where $W(x)$ is the hypergeometric function; see the Subsection \ref{hyper}. Three branches of the elliptic curve $xy(x-y)=\pi^{3}_{3}/27=5.513_{+}$ are shown; on this curve the expression for $\lambda(x,y)$ is particularly simple. Note however that there is a countable number of curves on which $\lambda(x,y)$ has double-zeros or poles, but these fall out of the borders of the picture.}
\label{figure2}
\end{figure}
Of course, one expression for $\lambda$ is sufficient, but interchanging $x$ and $y$ changes $\v$ into $-\v$, and we rather rewrote $\cm(-\v)$ and $\sm(-\v)$ in terms of $\cm(\v)$ and $\sm(\v)$ using Proposition \ref{dixonian}. The function $\lambda$ is unramified because of Proposition \ref{basic-sc}. Indeed, for given $(x,y)\in\mathbb{C}^{2}$, fix one of the three values for $\v=[xy(x-y)]^{1/3}$, and use this value consistently in the above formula for $\lambda(x,y)$. Then we get that the value of $\lambda(x,y)$ does not depend on our choice of the cubic root: we can rewrite the expression for $\lambda$ as
\begin{eqnarray*}
\lambda(x,y)=\frac{x(x-y)\big{(}c-(s\v^{-1})c^2y+(s\v^{-1})^2xy\big{)}^{2}}
{\big{(}x-c^3y\big{)}\big{(}c^2-(s\v^{-1})x+(s\v^{-1})^2cxy\big{)}},\\
\end{eqnarray*}
and both series $\cm(\v)$ and $\sm(\v)\v^{-1}$ contain only integral powers of $x,y$. This function does satisfy the property (PDE) by the very solution; see the Subsection \ref{subsol}. The property (FLOW), as in the cases of $\phi^{{\rm e}}$ and $\phi^{{\rm t}}$, is equivalent to addition formulas for the Dixonian functions given by Proposition \ref{dixonian}. This is, to our understanding, a remarkable fact! We note that it is considerably more natural to write down the analytic expression of $\lambda$ in terms of $A=\s(\v)\v$ and $B=\c(\v)\v$, where $\s$ and $\c$ are other two elliptic functions introduced in the Subsection \ref{sub-dix}. This expression is given by (\ref{quasii}), and the rational function in the middle of (\ref{quasii}) is $1$-homogenic and defines the quasi-rational projective flow of genus $1$; see the Subsection \ref{sub-qf}. We chose to present the results in terms of Dixonian functions $\sm$ and $\cm$ solely because these are more conventional functions and there are few results already available in the literature. In fact, the expression for $\lambda(x,y)$ in the Theorem \ref{thm2} is also $1$-homogenic, if $x,y,\v$ are given weight $1$, and $\cm(\v)$ and $\sm(\v)$ are both of weight $0$.
\subsection{Values on special curves}
\label{corro}
Now we are able to say more about the Taylor coefficients of the function $\lambda(x,-x)x^{-1}$, which were computed in (\cite{alkauskas2}, Section 4, Step III); see (\ref{skew}).
\begin{cor}
Let $(x,y)=(z,-z)$, choose $\v=-z\sqrt[3]{2}$. We have
\begin{eqnarray*}
\frac{\lambda(z,-z)}{z}=\frac{\sqrt[3]{2}\big{(}c\sqrt[3]{4}-sc^2\sqrt[3]{2}-s^2\big{)}^{2}}
{\big{(}1+c^3\big{)}\big{(}c^2\sqrt[3]{4}+s\sqrt[3]{2}-s^2c\big{)}},\quad c=\cm(-z\sqrt[3]{2}),\quad s=\sm(-z\sqrt[3]{2}).
\end{eqnarray*}
Further, $\lambda(z,-z)z^{-1}$ is an elliptic function.
In general, for fixed $x,y$, $xy(x-y)\neq 0$, the function $\lambda(xz,yz)z^{-1}$ is an elliptic function in $z$.
On the other hand, fix $c\neq 0$. On the elliptic curve $E(c)$ the function $\lambda(x,y)$ is a rational function in $x,y$ with degree $6$ numerator and degree $3$ denominator:
\begin{eqnarray*}
\lambda(x,y)\Big{|}_{E(c)}&=&\frac{x(x-y)(B-AB^{2}y+A^2xy)^{2}}
{(x-B^3y)(B^2-Ax+A^2Bxy)}=\mathscr{R}(A,B;x,y),\\
\lambda(y,x)\Big{|}_{E(c)}&=&\frac{y(x-y)(B^2-Ax+A^2Bxy)^{2}}
{(x-B^3y)(B-AB^2y+A^2xy)}
=\mathscr{R}\Big{(}\frac{A}{B},\frac{1}{B};y,x\Big{)}.
\end{eqnarray*}
here $A=\sm(\v)\v^{-1}$, $B=\cm(\v)$.
\label{isvada}
\end{cor}
The power series  at the origin for $\cm(-z\sqrt[3]{2})$ and $\sm(-z\sqrt[3]{2})\sqrt[3]{4}$ contain only rational coefficients, so Taylor coefficients of $\lambda(z,-z)z^{-1}$ belong to $\mathbb{Q}$, and they are given by (\ref{skew}). In fact, we can use the Taylor series in Proposition \ref{dixonian} to calculate the expansion of $\lambda(z,-z)z^{-1}$; these series were borrowed from \cite{flajolet-c}, formula (6), only note the typo ``8880000" instead of the correct value ``880000". This and our first method of calculations, i.e. (\ref{skew}), which is done using the recursion (\ref{upow}), match perfectly. \\

For some special $c$ the function $\lambda(x,y)$ has a particularly simple expression. For example, let $\v=(\frac{\pi_{3}}{3}+\pi_{3}k)$ for any $k\in\mathbb{Z}$, $c=\v^{3}$. (see the Subsection \ref{sub-dix}). Then $B=0$, $A=\v^{-1}$, and then
\begin{eqnarray*}
\lambda(x,y)\big{|}_{E(c)}=-\frac{x(x-y)y^2}{\v^{3}}=-y,\quad \lambda(y,x)\big{|}_{E(c)}=x-y.
\end{eqnarray*}
Also, for example, when $\v$ is a zero of $\sm$ equivalent to the point $q_{2,1}$ (see the Subsection \ref{sub-dix} as well), we get $A=0$, $B=\omega^2$, and then
\begin{eqnarray*}
\lambda(x,y)\big{|}_{E(c)}=x,\quad \lambda(y,x)\big{|}_{E(c)}=y.
\end{eqnarray*}
To summarize, we have observed a different behaviour of $\lambda(x,y)$ on three kinds of curves: lines through the origin; elliptic curves; transcendental curves parametrized by hypergeometric function $W(x)$. In the end of the Subsection \ref{sub:rat} we will see that the function $\lambda(x,y)$ is algebraic function on another transcendental curve
\begin{eqnarray*}
\s(u)u\equiv 1.
\end{eqnarray*}
\subsection{The symmetry property} 
\label{sub-symmetry}
Now we will turn our attention to the property (SYMM). Note that if $xy(x-y)=c$, then $(-y,x-y),(y-x,-x)\in E(c)$, and so the first equality in (\ref{period})
means that the following as if holds on the elliptic  curve $E(c)$:
\begin{eqnarray*}
\mathscr{R}(A,B;x,y)+\mathscr{R}(A,B;-y,x-y)
+\mathscr{R}(A,B;y-x,-x)\mathop{=}^{?}0.
\end{eqnarray*}
In fact, this is NOT the identity satisfied by the rational function
$\mathscr{R}(A,B;x,y)$ for arbitrary free $A,B,x,y$, as computer calculations show: if we calculate the l.h.s. for arbitrary unspecified variables $A,B,x,y$, we get a complicated and lengthy expression, and it involves high powers of $A$ and $B$. But in our case the four variables are related via the identity $A^{3}xy(x-y)+B^3\equiv 1$. Consider now the field of rational functions $\mathbb{C}(A,B,x,y)$. What we really get is
\begin{eqnarray*}
\mathscr{R}(A,B;x,y)+\mathscr{R}(A,B;-y,x-y)+\mathscr{R}(A,B;y-x,-x)
\equiv 0\text{ mod }\big{(}A^{3}xy(x-y)+B^3-1\big{)},
\end{eqnarray*}
and this is satisfied for free variables: the numerator of the l.h.s. is divisible by $A^{3}xy(x-y)+B^3-1$, what MAPLE really does confirm, too. Likewise, if $(x,y)\in E(c)$, then $(-x,y-x)\in E(-c)$. So, the second identity of (\ref{period}), if we use Proposition \ref{dixonian}, gives
\begin{eqnarray*}
\mathscr{R}(A,B;x,y)+\mathscr{R}\Big{(}\frac{A}{B},\frac{1}{B};-x,y-x\Big{)}
\equiv 0\text{ mod }\big{(}A^{3}xy(x-y)+B^3-1\big{)}.
\end{eqnarray*}
This is also easily reverified using computer algebra engine. \\

In the Subsection \ref{sub:group} we will see that the Symmetry property can be described in terms of the group structure of the elliptic curve $E(c)$.
\subsection{Relation to P\'{o}lya-Eggenberger urn model} The PrTE (\ref{funkk}) was first introduced in \cite{alkauskas1}. This is a new and fascinating object with many ramifications; see the Section \ref{sec:further}. While working on the paper \cite{alkauskas2} we were unaware of any research related to the algebraic theory of the PDE (\ref{g-parteq}). However, in the course of the evolvement of the current work the importance of the constant $\pi_{3}$ emerged. This later led us to the paper \cite{flajolet-b}, and then to many other works, including  \cite{dumont1,dumont2,flajolet-bl,flajolet-c,flajolet-g-p,fdp}. It appeared that a theory related to the PDE similar to (\ref{g-parteq}) was already well investigated with relation to the \emph{P\'{o}lya-Eggenberger urn models}. We will shortly describe the setting and the reasons why both directions (projective flows and urns) have a certain intersection in common.\\

Suppose there is an urn containing a finite number of balls of two types: ${\sf x}$ and ${\sf y}$. At each tick of a time we choose one ball at random, observe its color, put it back to the urn, and then add or subtract balls from the urn according to the result of this random pick. If the ball ${\sf x}$ is picked, we add balls of type ${\sf x}$, ${\sf y}$ in quantities $\alpha$ and $\beta$, respectively. If the ball is ${\sf y}$, we add balls of type ${\sf x}$, ${\sf y}$ in quantities $\gamma$ and $\delta$, and negative numbers are allowed and stand for a subtraction. The urn is said to be \emph{balanced}  if increase of balls at each step is deterministic, i.e. $\alpha+\beta=\gamma+\delta=s$. The task is to describe all histories of any given length; \emph{a history} is any legal sequence of steps that lead from the initial urn to any specific one; say, the one where all balls are of type ${\sf x}$. (Of course, certain simple arithmetic conditions, called \emph{tenability conditions}, on coefficients $\alpha,\beta,\gamma,\delta$, and the initial urn composition are needed to ensure that at each step the action of subtraction can be performed). It appears that histories can be investigated via analysis of the differential operator \cite{flajolet-g-p}
\begin{eqnarray*}
\Upsilon=x^{1+\alpha}y^{\beta}\frac{\partial}{\partial x}+x^{\gamma}y^{1+\delta}\frac{\partial}{\partial y}.
\end{eqnarray*}
This is completely analogous to the recurrence (\ref{upow}) where $\varpi(x,y)\bl\varrho(x,y)$ is an arbitrary $2$-homogenic vector field. So, it appears that the property of the urn to be \textbf{balanced} corresponds to \textbf{projectivity} in our setting. Nevertheless, these fields of research have many differences. While we deal with arbitrary $2$-homogenic rational functions in $\mathbb{C}(x,y)$, the research in P\'{o}lya urns concentrates and has a combinatoric meaning only for the monomials $x^{1+\alpha}y^{\beta}\bl x^{\gamma}y^{1+\delta}$. On the other hand, our vector fields are $2$-homogenic, which corresponds to the balance $1$ in the urn setting; but other balances are also investigated and are of great importance; balance $s\neq 1$ does not seem to correspond to a flow or other geometric object. So, the two theories are different though have an non-empty intersection and there are many similarities! For example, a relative of the function $\lambda(x,y)$  has already appeared in the literature, though in a disguise which at first seem to be unrecognisable (\cite{flajolet-c}, Proposition 3):
\begin{eqnarray*}
H(x,z)=(1-x^{3})^{1/3}\,{\rm smh}\Bigg{(}(1-x^{3})^{1/3}z+
\int_{0}^{x(1-x^{3})^{-1/3}}\frac{\d s}{(1+s^{2})^{2/3}}\Bigg{)};
\end{eqnarray*}
here, as mentioned in the Subsection \ref{sub-dix}, ${\rm smh}(u)=-\sm(-u)$. Still, this function is closely related to $\lambda(x,y)$, which can be seen if we go a few step further: that is, apply the addition formulas for the Dixonian elliptic functions and use their relation with hypergeometric series. Then the above complicated formula would turn into a formula similar to that in our Theorem \ref{thm2}. Philippe Flajolet and many other mathematicians have developed powerful analytic and combinatoric techniques to investigate (not necessarily balanced) urns; see, for example, the excellent study \cite{flajolet-bl}. In our paper we did not touch the combinatoric aspects at all. Possibly, these techniques will be of big help in the prospective work on quasi-rational flows \cite{alkauskas4}.
\section{The proofs. Further properties}
\subsection{The analytic expression for $\lambda(x,y)$}
\label{subsol}
With all tools at hand, we can solve analytically the PDE (\ref{parteq}) and prove the Theorem \ref{thm2}. 
\begin{proof}Let us introduce
\begin{eqnarray*}
\Omega(u,\v)=\lambda\Big{(}\s(u)\v,\c(u)\v\Big{)}.
\end{eqnarray*}
Each pair of finite complex numbers $x,y$ can be represented as $(x,y)=(\s(u)\v,\c(u)\v)$ for $\v\in\mathbb{C}$, $u\in\mathcal{F}$ (fundamental parallelogram of our elliptic functions), except when $xy(x-y)=0$.
The boundary condition requires that
\begin{eqnarray}
\lim\limits_{\v\rightarrow 0}\frac{\Omega(u,\v)}{\v}=\s(u).
\label{mod-bound}
\end{eqnarray}
Then, by a direct calculation and Proposition \ref{basic-sc},
\begin{eqnarray*}
\Omega_{u}(u,\v)&=&\lambda_{x}\Big{(}\s(u)\v,\c(u)\v\Big{)}\s'(u)\v
+\lambda_{y}\Big{(}\s(u)\v,\c(u)\v\Big{)}\c'(u)\v\\
&=&\lambda_{x}\Big{(}\s(u)\v,\c(u)\v\Big{)}\Big{(}-\s^{2}(u)+2\s(u)\c(u)\Big{)}\v\\
&+&\lambda_{y}\Big{(}\s(u)\v,\c(u)\v\Big{)}\Big{(}-\c^{2}(u)+2\s(u)\c(u)\Big{)}\v;\\
\Omega_{\v}(u,\v)&=&\lambda_{x}\Big{(}\s(u)\v,\c(u)\v\Big{)}\s(u)
+\lambda_{y}\Big{(}\s(u)\v,\c(u)\v\Big{)}\c(u).
\end{eqnarray*}
Therefore, according to (\ref{parteq}), we have
\begin{eqnarray*}
\v\Omega_{u}(u,\v)+\v\Omega_{\v}(u,\v)=\Omega(u,\v).
\end{eqnarray*}
The general solution of this linear PDE is given by $H(u-\v)\v$, for any $H$, which is smooth in $\mathbb{C}$, except possibly at the poles. The boundary condition (\ref{mod-bound}) requires that $H(u)=\s(u)$.
So,
\begin{eqnarray*}
\Omega(u,\v)=\lambda\Big{(}\s(u)\v,\c(u)\v\Big{)}=\s(u-\v)\v.
\end{eqnarray*}
Recall that $x=\s(u)\v$, $y=\c(u)\v$. Then $xy(x-y)=\v^3$, and thus $\v=[xy(x-y)]^{1/3}$.
Let $S=\s(u)$, $C=\c(u)$, $a=\s(-\v)$, $b=\c(-\v)$. Then using Proposition \ref{basic-sc}, we get
\begin{eqnarray*}
\s(u-\v)=\frac{(b+C-abSC)^{2}aS}{(1-a^2bS^2C)(aSC+abS-1)}.
\end{eqnarray*}
Replace in the latter expression $S=x\v^{-1}$, $C=y\v^{-1}$. We get
\begin{eqnarray*}
\s(u-\v)=\frac{(b\v^2+y\v-abxy)^{2}ax}{(\v^3-a^2bx^2y)(axy+abx\v-\v^2)}.
\end{eqnarray*}
Further, let $A=\s(\v)$, $B=\c(\v)$. Again, from Proposition \ref{basic-sc} we know that $a=A$, $b=(AB)^{-1}$. Rewriting the above in terms of $A$, $B$, we get
\begin{eqnarray}
\s(u-\v)=\frac{(\v^2+ABy\v-Axy)^{2}x}{A(B\v^3-Ax^2y)(ABxy+x\v-B\v^2)}.
\label{analo}
\end{eqnarray}
Finally, let $\sm(\v)=s$, $\cm(\v)=c$. So, $A=-s^{2}c^{-1}$, $B=c^2s^{-1}$. This gives
\begin{eqnarray*}
\s(u-\v)\v=\frac{(c\v^2-sc^2y\v+s^2xy)^{2}x\v}
{(c^3\v^3+s^3x^2y)(s^2cxy-sx\v+c^2\v^2)}.
\end{eqnarray*}
Since $x^2y-xy^2=\v^3$, the last simplifies to
\begin{eqnarray*}
\lambda(x,y)=\s(u-\v)\v=\frac{\v(c\v^2-sc^2y\v+s^2xy)^{2}}
{y(x-c^3y)(c^2\v^2-sx\v+s^2cxy)}.
\end{eqnarray*}
This is exactly the statement of Theorem \ref{thm2}.\end{proof}
\subsection{Rational and algebraic interpretation of $\lambda$}
\label{sub:rat}
We now proceed in showing that the function $\lambda(x,y)$ can be described purely algebraically without appeal to elliptic functions, which, most importantly, leads us to the definition of projective quasi-rational flows of arbitrary genus; see the Subsection \ref{sub-qf}. Let $\v$ be as in the previous Subsection, but let us rewrite the formula (\ref{analo}) not in terms of $A$ and $B$ as given there, but (as mentioned in the end of the Subsection \ref{sub:ann}) in terms of $A=\s(\v)\v$, $B=\c(\v)\v$, which is the most natural. We thus have
\begin{eqnarray}
\lambda(x,y)=\frac{[x(x-y)+AB-Ax]^2y(x-y)}
{A[B(x-y)-Ax][AB+x(x-y)-B(x-y)]}=\mathscr{T}(A,B;x,y).
\label{quasii}
\end{eqnarray}
Note that if $A,B,x,y$ are all given a weight $1$, then $\mathscr{T}(A,B;x,y)$ is a $1$-homogenic function. 
The boundary condition property (\ref{boundd}) reads as
\begin{eqnarray}
\lim\limits_{z\rightarrow 0}\frac{\mathscr{T}(-z^{3}xy(x-y),1;xz,yz)}{z}=x.\label{tbound}
\end{eqnarray}
The PDE (\ref{parteq}) can be rewritten in terms of the rational function $\mathscr{T}$. Indeed,
\begin{eqnarray*}
\v_{x}=\frac{2xy-y^2}{3\v^2},& & \v_{y}=\frac{x^2-2xy}{3\v^2},\\
A_{x}=[\s'(\v)\v+\s(\v)]\v_{x},& & B_{x}=[\c'(\v)\v+\c(\v)]\v_{x},\\
A_{y}=[\s'(\v)\v+\s(\v)]\v_{y},& & B_{y}=[\c'(\v)\v+\c(\v)]\v_{y}.
\end{eqnarray*}
Thus,
\begin{eqnarray}
x\v_{x}+y\v_{y}=\v,\quad \v_{x}(x^2-2xy)+\v_{y}(y^2-2xy)=0.
\label{homo}
\end{eqnarray}
The PDE (\ref{parteq}) now read as
\begin{eqnarray*}
0&=&\lambda_{x}(x^2-2xy-x)+\lambda_{y}(y^2-2xy-y)+\lambda(x,y)\\
&=&\mathscr{T}_{A}[\s'(\v)\v+\s(\v)]\v_{x}(x^2-2xy-x)+\mathscr{T}_{B}[\c'(\v)\v+\c(\v)]\v_{x}(x^2-2xy-x)\\
&+&\mathscr{T}_{x}(x^2-2xy-x)\\
&+&\mathscr{T}_{A}[\s'(\v)\v+\s(\v)]\v_{y}(y^2-2xy-y)+\mathscr{T}_{B}[\c'(\v)\v+\c(\v)]\v_{y}(y^2-2xy-y)\\
&+&\mathscr{T}_{y}(y^2-2xy-y)+\mathscr{T}\\
&=&-\mathscr{T}_{A}[\s'(\v)\v+\s(\v)]\v-\mathscr{T}_{B}[\c'(\v)\v+\c(\v)]\v\\
&+&\mathscr{T}_{x}(x^2-2xy-x)+\mathscr{T}_{y}(y^2-2xy-y)+\mathscr{T}\\
&=&\mathscr{T}_{A}(A^2-2AB-A)+\mathscr{T}_{B}(B^2-2AB-B)\\
&+&\mathscr{T}_{x}(x^2-2xy-x)+\mathscr{T}_{y}(y^2-2xy-y)+\mathscr{T};
\end{eqnarray*}
in the penultimate equality we used (\ref{homo}), and in the last one Proposition \ref{basic-sc} was used.
We get that the PDE for the rational function $\mathscr{T}$ is just an amalgam of two identical PDE's. Each of them separately is solvable in terms of elliptic functions, while their amalgam gives a rational solution! Since $\mathscr{T}$ is $1$-homogenic, we seem to obtain the $4$-variable PDE 
\begin{eqnarray}
\mathscr{T}_{A}(A^2-2AB)+\mathscr{T}_{B}(B^2-2AB)
+\mathscr{T}_{x}(x^2-2xy)+\mathscr{T}_{y}(y^2-2xy)\mathop{=}^{?}0.\label{quaq}
\end{eqnarray}
However, computer calculations show that this is NOT the identity satisfied by the rational function $\mathscr{T}$ - that is the meaning of the question mark. In fact, the four variables $A,B,x,y$ are not independent but satisfy $AB(A-B)=xy(x-y)$. Now, calculations prove that this PDE is indeed satisfied modulo $AB(A-B)-xy(x-y)$. We thus get the needed property: for free variables $A,B,x,y$, one has
\begin{eqnarray*}
\text{l.h.s. of }(\ref{quaq})\equiv 0\text{ mod }\Big{(}AB(A-B)-xy(x-y)\Big{)}.
\end{eqnarray*}
This shows that the first coordinate of the flow $\Lambda$, now in the avatar $\mathscr{T}$, can be described purely algebraically as follows: it is $1$-homogenic, it satisfies the boundary condition (\ref{tbound}), and the $4$-variable PDE in the quotient ring. This leads us to the definition of quasi-rational flows in the Subsection  \ref{sub-qf}.\\

Moreover, we can give (\ref{quaq}) another appearance by eliminating one variable; say, $A$. First, by the homogeneity property,
\begin{eqnarray*} 
A\mathscr{T}_{A}=\mathscr{T}-B\mathscr{T}_{B}-x\mathscr{T}_{x}-y\mathscr{T}_{y}.
\end{eqnarray*}
Now plug this into (\ref{quaq}). The derivative with respect to the variable $A$ is not present any more, and we can freely choose $A=1$. We thus obtain the equation
\begin{eqnarray} 
\mathscr{T}_{B}(3B^2-3B)+\mathscr{T}_{x}(x^2-2xy-x+2Bx)
+\mathscr{T}_{y}(y^2-2xy-y+2By)=(2B-1)\mathscr{T},
\label{tfun}
\end{eqnarray}
which is valid in the quotient ring $\mathbb{C}(B;x,y)/\big{(}B(1-B)-xy(x-y)\big{)}$. So, let $B=B(x,y)$ be the solution of $B(1-B)=xy(x-y)$, which we can express in quadratic radicals:
\begin{eqnarray*}
B(x,y)=\frac{1}{2}+\frac{1}{2}\sqrt{1-4xy(x-y)}.
\end{eqnarray*}
Let therefore
\begin{eqnarray*}
\mathcal{E}(x,y)=\mathscr{T}(1,B(x,y);x,y)
\end{eqnarray*}
be an algebraic function of two free variables. This also satisfies the PDE, as we will now see. First, we have
\begin{eqnarray*}
B_{x}=\frac{2xy-y^2}{1-2B},\quad B_{y}=\frac{x^2-2xy}{1-2B}.
\end{eqnarray*}
Therefore,
\begin{eqnarray*}
\mathcal{E}_{x}=\mathscr{T}_{B}\frac{2xy-y^2}{1-2B}+\mathscr{T}_{x},\quad
\mathcal{E}_{y}=\mathscr{T}_{B}\frac{x^2-2xy}{1-2B}+\mathscr{T}_{y}.
\end{eqnarray*}
So,
\begin{eqnarray*}
\mathcal{E}_{x}(x^2-2xy)+\mathcal{E}_{y}(y^2-2xy)&=&\mathscr{T}_{x}(x^2-2xy)+\mathscr{T}_{y}(y^2-2xy),\\
\mathcal{E}_{x}x(1-2B)+\mathcal{E}_{y}y(1-2B)&=&\mathscr{T}_{x}x(1-2B)+\mathscr{T}_{y}y(1-2B)+
\mathscr{T}_{B}(3x^2y-3xy^2).
\end{eqnarray*}
Subtract now second from the first. This, using (\ref{tfun}), implies (in an unconventional form)
\begin{eqnarray*} 
\frac{\mathcal{E}_{x}(x^2-2xy)+\mathcal{E}_{y}(y^2-2xy)}{\mathcal{E}-x\mathcal{E}_{x}-y\mathcal{E}_{y}}=2B-1=\sqrt{1-4xy(x-y)}.
\end{eqnarray*}
The validity of this identity was also double-checked on MAPLE, and it holds true; both numerator and denominator on the left side are rather complicated expressions.
\subsection{The group structure of $E(c)$ and its effect on the flow $\Lambda$}
\label{sub:group}
 Let, as before, $\omega=e^{2\pi i/3}$. If $c\neq 0$, the projective transformation
\begin{eqnarray*}
X=cr-q,\quad Y=-cr-q,\quad Z=2p
\end{eqnarray*}
maps the curve $XY(X-Y)=cZ^3$ to the elliptic curve \cite{knapp}
\begin{eqnarray*}
\widehat{E}(c)=\{(p:q:r)\in P^{2}(\mathbb{C}):q^2r=4p^3+c^2r^3\}.
\end{eqnarray*}
If $\mathbf{P}=(X:Y:1)\in E(c)$, then $-\mathbf{P}=(-Y:-X:1)$.
Further, let $\m{P}_{1}=(X_{1}:Y_{1}:1)$ and $\m{P}_{2}=(X_{2}:Y_{2}:1)$ be two (finite) points on $E(c)$. The standard addition formulas for the curve $\widehat{E}(c)$ translate onto $E(c)$ as follows. Let $\m{P}_{3}=\m{P}_{1}+\m{P}_{2}=(X_{3}:Y_{3}:Z_{3})$. Then
\begin{eqnarray*}
X_{3}&=&(-X_{1}Y_{1}+X_{1}^{2}+X_{2}Y_{2}-X_{2}^{2})(Y_{1}-Y_{2})^{2},\\
Y_{3}&=&(X_{1}Y_{1}-Y_{1}^{2}-X_{2}Y_{2}+Y_{2}^{2})(X_{1}-X_{2})^2,\\
Z_{3}&=&(X_{1}-X_{2})(Y_{1}-Y_{2})(X_{1}-Y_{1}+Y_{2}-X_{2}).
\end{eqnarray*}
This can be given an alternative expression. Using $X_{1}Y_{1}(X_{1}-Y_{1})=c$ and  $X_{2}Y_{2}(X_{2}-Y_{2})=c$, we can rewrite this as
\begin{eqnarray*}
X_{3}&=&(Y_{1}-Y_{2})^{3}X_{1}X_{2},\\
Y_{3}&=&(X_{1}-X_{2})^3Y_{1}Y_{2},\\
Z_{3}&=&(X_{1}-X_{2})(Y_{1}-Y_{2})(X_{1}Y_{1}-X_{2}Y_{2}).
\end{eqnarray*}
The duplication formula read as follows. If $\m{P}=(X:Y:1)$, then $2\m{P}=(X_{1}:Y_{1}:Z_{1})$, where
\begin{eqnarray*}
X_{1}&=&(2X-Y)^{3}Y,\\
Y_{1}&=&(2Y-X)^{3}X,\\
Z_{1}&=&(X+Y)(2X-Y)(2Y-X).
\end{eqnarray*}
Now we will introduce few special points, which give the cyclic $6$-group on the curve $E(c)$, as presented in the Table \ref{table2}.
\begin{table}[h]
\begin{tabular}{|| c | c l| c l ||}
Order & $\widehat{E}(c)$ & $(p:q:r)$ & $E(c)$ & $(X:Y:Z)$\\
1 & $\widehat{\m{O}}$ & $(0:-1:0)$     & $\m{O}$ & $(1:1:0)$ \\
2&$\widehat{\m{Q}}_{2}$ & $(-\sqrt[3]{c^2/4}:0:1)$
&$\m{Q}_{2}$ & $(-\sqrt[3]{c/2}:\sqrt[3]{c/2}:1)$.\\
3 & $\widehat{\m{Q}}_{3}$ & $(0:c:1)$   & $\m{Q}_{3}$ & $(0:1:0)$\\
3 & $2\widehat{\m{Q}}_{3}$ & $(0:-c:1)$ & $2\m{Q}_{3}$ & $(1:0:0)$.\\
6 & $\widehat{\m{Q}}_{6}$ & $(\sqrt[3]{2c^2}:3c:1)$ &$\m{Q}_{6}$ & $(-\sqrt[3]{c/2}:-\sqrt[3]{4c}:1)$. \\
6 & 5$\widehat{\m{Q}}_{6}$ & $(\sqrt[3]{2c^2}:-3c:1)$ &$5\m{Q}_{6}$ &
$(\sqrt[3]{4c}:\sqrt[3]{c/2}:1)$.
\end{tabular}
\caption{Finite order points on elliptic curves $\widehat{E}(c)$ and $E(c)$.}
\label{table2}
\end{table}
Note that for any $c\neq 0$ these $6$ points are different. The addition formula below give
\begin{eqnarray*}
\m{Q}_{6}+\m{Q}_{2}=2\m{Q}_{3},\quad 5\m{Q}_{6}+\m{Q}_{2}=\m{Q}_{3},\quad \m{Q}_{3}=2\m{Q}_{6}. 
\end{eqnarray*} 
So, we deduce that $\m{Q}_{2}=3\m{Q}_{6}$. Let this group be $C_{6}$. When $c$ is complex, we fix the cubic root $\sqrt[3]{c}$.  We can also consider a wider group. Let $\m{Q}^{\omega}_{6}=(-\omega\sqrt[3]{c/2}:-\omega\sqrt[3]{4c}:1)$, $\m{Q}^{\omega^{2}}_{6}=(-\omega^{2}\sqrt[3]{c/2}:-\omega^{2}\sqrt[3]{4c}:1)$. The addition formulas give
\begin{eqnarray*}
\m{Q}_{6}+\m{Q}^{\omega}_{6}+\m{Q}^{\omega^{2}}_{6}&=&\m{O},\\
\m{Q}_{6}-\m{Q}^{\omega}_{6}&=&\m{Q}^{\omega^{2}}_{2}=(-\omega^{2}\sqrt[3]{c/2}:\omega^{2}\sqrt[3]{c/2}:1),\\
\m{Q}_{6}-\m{Q}^{\omega^{2}}_{6}&=&\m{Q}^{\omega}_{2}=(-\omega\sqrt[3]{c/2}:\omega\sqrt[3]{c/2}:1).
\end{eqnarray*}
So, these points produce the group $C_{12}$ which is isomorphic to $\mathbb{Z}_{6}\times\mathbb{Z}_{2}$ via the following (non-canonical) isomorphism:
\begin{eqnarray*}
\m{Q}_{6}\mapsto(1,0),\quad \m{Q}^{\omega}_{6}\mapsto(1,1),\quad
\m{Q}^{\omega^{2}}_{6}\mapsto (4,1),\quad \m{Q}^{\omega^{2}}_{2}\mapsto (0,1),\quad
\m{Q}^{\omega}_{2}\mapsto (3,1).
\end{eqnarray*}
From the property (ELL) we know that if $\m{P}=(x:y:1)\in E(c)$, then 
\begin{eqnarray*}
\Lambda^{z}(\mathbf{P})=(\lambda^{z}(x,y):\lambda^{z}(y,x):1)\in E(c).
\end{eqnarray*}
Further, for the three points at infinity - $\m{O}$, $\m{Q}_{3}$ and $2\m{Q}_{3}$ - the action of $\Lambda$ can be calculated using the Corollary \ref{isvada}. Here we use its notation. Indeed, assume $A,B\neq 0$, $A^{3}c+B^{3}=1$. Let  now $(x,y)\in E(c)$, and $x,y\rightarrow\infty$ while remaining on the curve (in case $c>0$ is real, this corresponds to the left-bottom asymptote in the Figure \ref{figure2}). Then $(x:y:1)\rightarrow\m{O}$. Thus, using the identity $(x-y)=c(xy)^{-1}$, we can calculate
\begin{eqnarray*}
\Lambda(\m{O})=\lim_{\substack{x\sim\,y\rightarrow\infty\\x-y\sim\,cx^{-2}}}
\mathscr{R}(A,B;x,y)\bl\mathscr{R}\Big{(}\frac{A}{B},\frac{1}{B};y,x\Big{)}=\frac{1}{AB}\bl \frac{B^2}{A}=\m{F}_{c}\in E(c),
\end{eqnarray*}
where $\m{F}_{c}$ is the special point on the curve $E(c)$ as definition of $A$ and $B$ shows. In the same manner,
\begin{eqnarray*}
\Lambda(\m{Q}_{3})=-\frac{cA^{2}}{B}\bl -\frac{1}{AB},\quad \Lambda(2\m{Q}_{3})=-\frac{B^2}{A}\bl \frac{cA^2}{B}.
\end{eqnarray*}
In case $A=0$ the calculations give
\begin{eqnarray*}
\Lambda(\m{O})=\m{O},\quad \Lambda(\m{Q}_{3})=\m{Q}_{3},\quad \Lambda(2\m{Q}_{3})=2\m{Q}_{3},
\end{eqnarray*}
and in case $B=0$ we have
\begin{eqnarray*}
\Lambda(\m{O})=2\m{Q}_{3},\quad \Lambda(\m{Q}_{3})=\m{O},\quad \Lambda(2\m{Q}_{3})=\m{Q}_{3}.
\end{eqnarray*}
The order $3$ linear map $(x,y)\mapsto (y-x,-x)$ corresponds to $\m{P}\mapsto\m{P}+\m{Q}_{3}$. If we know already that $\Lambda(\m{x})=(\lambda(x,y),\lambda(y,x))$, the symmetry property (SYMM) now can be checked to be equivalent to the following nice identity:
\begin{eqnarray*}
\Lambda(\m{P}+\m{Q}_{3})=\Lambda(\m{P})+\m{Q}_{3}\text{ for every }\m{P}\in E(c).
\end{eqnarray*}
\subsection{The case of a vector field $\varpi(x,y)\bl 0$}
\label{sub-zero}
In this subsection we will show how to find all unramified flows in case in the PDE (\ref{g-parteq}) we have $\varrho(x,y)\equiv 0$; see the first paragraph of the Subsection \ref{sub-quadratic} for the explanation. It appears that no new unramified flows appear, apart from those found in \cite{alkauskas2}, Chapter 4, Step II. It was shown there that the solution $u(x,y)$ to (\ref{g-parteq}) in case $\varrho(x,y)\equiv 0$ satisfies
\begin{eqnarray*}
\int\limits_{y/u(x,y)}^{y/x}\frac{\d t}{\varpi(1,t)}=y.
\end{eqnarray*} 
Put $y=1$, and let $q(x)=\frac{1}{u(x^{-1},1)}$. Then (after a change $x\mapsto x^{-1}$) we get that one needs to find all rational functions $\varpi(1,t)$ such that the function $q(x)$, $x\in\mathbb{C}$, defined by 
\begin{eqnarray*}
\int\limits_{q(x)}^{x}\frac{\d t}{\varpi(1,t)}=1,
\end{eqnarray*} 
is single-valued. This is equivalent to the following: for all 
$A,B\in\mathbb{C}$ satisfying
\begin{eqnarray}
\int\limits_{A}^{B}\frac{\d t}{\varpi(1,t)}=0,
\label{lygybe}
\end{eqnarray}
this condition forces $A=B$; the integral is taken via any path only avoiding singularities. Suppose $\varpi(1,t)=k(t)\ell(t)^{-1}$, where $k(t),\ell(t)$ are polynomials and the fraction is irreducible over $\mathbb{C}$. If $\deg(\ell)>0$, let $t_{0}$ be the root of $\ell(t)$, say, of multiplicity $n_{0}\geq 1$. Let us define
\begin{eqnarray*}
T(x)=\int\limits_{t_{0}}^{x}\frac{\d t}{\varpi(1,t)},
\end{eqnarray*}
where $|x-t_{0}|<\epsilon$, $\epsilon$ is sufficiently small such that $\varpi(1,t)^{-1}$ has no poles or zeros inside this disc except at the center, and the path of integration is a segment. Then the function $T(x)$ is well defined and, most importantly, it has a zero at $x=t_{0}$ of multiplicity $n_{0}+1\geq 2$. So, we know that there exists a small open set $U\subset\{x:|x-t_{0}|<\epsilon\}$ containing $t_{0}$ such that in $U$ the function $T(x)$ attains each value exactly $n_{0}+1$ times, $T(t_{0})=0$ being the only multiple value \cite{shabat}. In particular, if $T(A)=T(B)$, $A\neq B$, $A,B\in U\setminus\{t_{0}\}$, then (\ref{lygybe}) holds; the path of integration consists of a junction of two segments $[A,t_{0}]$ and $[t_{0},B]$. We get a contradiction. Suppose now $\ell(t)\equiv 1$, and so $\varpi(1,t)$ is a polynomial. If we perform the change of variables $t\mapsto \frac{1}{t}$ in (\ref{lygybe}), we would obtain (by what was just proved) the neccessary condition that $\varpi(1,\frac{1}{t})t^2$ is also a polynomial. So, $\deg\varpi(1,t)\leq 2$. In \cite{alkauskas2} we have explored all these cases; in particular, $\varpi(1,t)=t$ and $\varpi(1,t)=t^2+1$ produce the unramified flows $\phi^{\rm exp}$ and $\phi^{{\rm tan}}$, the case $\varpi(1,t)=t+1$ produces the flow linearly conjugate to $\phi^{{\rm tan}}$, while $\varpi(1,t)=1$ and $\varpi(1,t)=t^2$ give rational flows.
\subsection{Other two elliptic flows}
\label{two-other}
It can be seen that all the results, apart from the $\Lambda(x,y)$-specific symmetry property (SYMM), can be restated in terms of the flows with the vector fields $x^2-3xy\bl y^2-3xy$ and $x^2-xy\bl y^2-2xy$, respectively; see Theorem \ref{thm1}, items 5) and 6), and also the Subsection \ref{sub-quadratic}. Consider the first vector field. Let, as already defined in Theorem \ref{thm1}, the solution to (\ref{g-parteq}) with the boundary condition (\ref{bound}) be given by $\psi(x,y)\bl\psi(y,x)$. The orbits of this flow are the curves $xy(x-y)^2=c$. Consider the projective curve $xy(x-y)^2=cz^4$, $c\neq 0$. If we make a birational change  $(x:y:z)\mapsto(\frac{z^2}{x}+y:y:z)$, we see that these curves are birationally equivalent to elliptic curves. So, let us introduce two new elliptic functions $p(u)$ and $q(u)$ with such properties. The Taylor series for both these functions start at $p(u)=u^{3}+\cdots$ and $q(u)=u^{-1}+\cdots$, and they contains only the powers $u^{4n-1}$, $n\geq 1$ and $n\geq 0$, respectively. Further, we require
\begin{eqnarray*}
\left\{\begin{array}{c@{\qquad}c}
1\equiv p(u)q(u)[p(u)-q(u)]^2,\\
p'(u)=-p^{2}(u)+3p(u)q(u),\\
q'(u)=-q^{2}(u)+3p(u)q(u).
\end{array}\right.
\end{eqnarray*} 
Using these properties, one can recurrently calculate the unique Taylor coefficients, and we do it with the help of MAPLE:
\begin{eqnarray*}
p(u)&=&u^3+\frac{1}{5}u^{7}+\frac{2}{25}u^{11}+\frac{127}{4875}u^{15}
+\frac{246}{27625}u^{19}+\frac{1246}{414375}u^{23}
+\frac{1234412}{1212046875}u^{27}+\cdots,\\
q(u)&=&u^{-1}+\frac{3}{5}u^{3}+\frac{17}{75}u^{7}+\frac{126}{1625}u^{11}
+\frac{32639}{1243125}u^{15}+\frac{6138}{690625}u^{19}
+\frac{42898}{14259375}u^{23}+\cdots.\\
\end{eqnarray*}
Let $x=p(u)\v$, $y=q(u)\v$, $\v=[xy(x-y)^2]^{1/4}$. Similarly as in the Subsection \ref{subsol}, we find that the solution of (\ref{g-parteq}) in case of the vector field $x^2-3xy\bl y^2-3xy$ satisfies
\begin{eqnarray*}
\psi(x,y)=\psi\Big{(}p(u)\v,q(u)\v\Big{)}=p(u-\v)\v.
\end{eqnarray*}
Let $A=p(\v)\v$, $B=q(\v)\v$. We can act completely analogously as in the Subsection \ref{subsol}, i.e. one can derive addition formulas for the elliptic functions $p$ and $q$ and then use them in the above expression. This gives the analogue of Theorem \ref{thm2} in this case, and we omit the details. \\

Equally, the vector field $x^2-xy\bl x^2-2xy$ has projective curves $(3x-2y)x^3y^2=cz^6$ as orbits. The projective birational change $(x:y:z)\mapsto (\frac{z^2}{x}:\frac{x^2}{y}:z)$ transforms this curve into $3z^2y-2x^3=cy^3$, and thus for $c\neq 0$ this is an elliptic curve, and so we can derive analogous results as in the other two cases.
\section{Open directions, further advances}
\label{sec:further} 
We finish this study with listing the variety of ways and points of view the multivariate PrTE can be explored from. If we pose a problem, it does not necessarily mean that this problem is hard - though it might be - or has a positive answer; the aim of this section is mainly expository: we just wish to exhibit the richness and structural variety that underlies (\ref{funkk}).
\subsection{Higher dimensional generalization of $\lambda$}\label{dim}
We will now define an analogue of the flow $\Lambda(\m{x})$ in higher dimensions. Let $N\in\mathbb{N}$, $N\geq 2$. The symmetric group $S_{N+1}$ has the standard $(N+1)$-dimensional permutation representation inside the group ${\rm GL}(\mathbb{C}^{N+1})$. The invariant subspace of this representation is the line $\mu\cdot(1,1,\ldots,1)$, $\mu\in\mathbb{C}$. So, $S_{N+1}$ acts on the orthogonal complement $\mathcal{A}\subset\mathbb{C}^{N+1}$ of this line which is given by
\begin{eqnarray*}
\mathcal{A}=\Big{\{}(v_{1},v_{2},\ldots, v_{N+1})\in\mathbb{C}^{N+1}:
\sum\limits_{i=1}^{N+1}v_{i}=0\Big{\}}.
\end{eqnarray*}
Let $\pi:S_{N+1}\mapsto{\rm GL}(\mathcal{A})$
be this representation. It is well known that it is irreducible and exact. Let us choose the basis of $\mathcal{A}$ as follows:
\begin{eqnarray*}
\mathbf{q}_{i}=(0,\ldots,0,\mathop{1}_{i},0,\ldots,0,-1),\quad
1\leq i\leq N.
\end{eqnarray*}
Each transposition of the form $(ij)\in S_{N+1}$, $1\leq i<j\leq N$, acts on vectors $\mathbf{q}_{i},1\leq i\leq N$, as the transposition $(\mathbf{q}_{i}\mathbf{q}_{j})$. On the other hand, if $\eta\in S_{N+1}$ is the transposition $(1(N+1))$, then the matrix representation of the linear map $\pi(\eta)$ in basis $(\mathbf{q}_{1},\ldots,\mathbf{q}_{N})^{T}$ is as follows:
\begin{eqnarray*}
\kappa=\begin{pmatrix}
-1 &   &   &   &  \\
-1 & 1 &   &   &  \\
-1 &   & 1 &   &  \\
\,\,\,\vdots & &   & \ddots &  \\
-1 &   &  &  & 1
\end{pmatrix};
\end{eqnarray*}
(matrices act on vector-columns by multiplication from the left). For example, when $N=2$, this corresponds to the matrix $\sigma\tau\sigma$, see (SYMM). Let $\Sigma_{N+1}=\pi(S_{N+1})$. For each $\gamma\in\Sigma_{N+1}$, we henceforth consider $\gamma$ as a matrix in the fixed basis $\{\mathbf{q}_{1},\ldots,\mathbf{q}_{N}\}$. Now, we want to find an $N$-tuple of quadratic forms
\begin{eqnarray*}
\mathbf{Q}(\m{x})=\big{(}Q_{1}(\mathbf{x}),Q_{2}(\mathbf{x}),\ldots,Q_{N}(\mathbf{x})\big{)},
\end{eqnarray*}
such that if $\Gamma=\Sigma_{N+1}$, we have
\begin{eqnarray}
\gamma^{-1}\circ\mathbf{Q}\circ\gamma(\mathbf{x})=
\mathbf{Q}(\mathbf{x})\label{kappa}
\end{eqnarray}
for each $\gamma\in\Gamma$. It is clear from the above that the vector $\mathbf{Q}$ should be symmetric with respect to all coordinates. This means
\begin{eqnarray}
\left\{\begin{array}{@{\qquad}l} Q_{i}(x_{1},\ldots,x_{i},\ldots,x_{N})=
Q_{1}(x_{i},\ldots,x_{1},\ldots,x_{N}),\quad 1\leq i\leq N,
\\ Q_{j}(x_{1},\ldots,x_{i},\ldots,x_{N})=
Q_{j}(x_{i},\ldots,x_{1},\ldots,x_{N}),\text{ for }j\neq 1,i. \end{array}\right.\label{symm}
\end{eqnarray}
Since transpositions $(ij)$, $1\leq i<j\leq N$, together with the transposition $(1(N+1))$ generate the whole group $S_{N+1}$, it is enough to find a quadratic form $Q_{1}(x_{1},\ldots,x_{N})$ such that if $Q_{i}$ are defined by the first entry of (\ref{symm}) and they all satisfy the second entry, then the vector $\mathbf{Q}$ satisfies (\ref{kappa}) in a special case $\gamma=\kappa$. This is a linear algebra task, and we find that the solution is given by
\begin{eqnarray*}
Q_{1}(\m{x})=x_{1}^{2}-\frac{2}{N-1}\cdot x_{1}
\sum\limits_{i=2}^{N}x_{i}.
\end{eqnarray*}So, we pose the following
\begin{prob}
Let $Q_{1}$ be as above, and $Q_{i}$ are given by (\ref{symm}). Let $\Lambda_{N}(\m{x})=(\lambda^{(1)},\ldots,\lambda^{(N)})$, where $\lambda^{(i)}(\m{x})$ is the solution of the PDE
\begin{eqnarray*}
\sum\limits_{i=1}^{N}f_{x_{i}}(\m{x})[Q_{i}(\m{x})-x_{i}]=-f(\m{x})\label{parteqN}
\end{eqnarray*}
with the boundary condition
\begin{eqnarray*}
\lim\limits_{z\rightarrow 0}\frac{f(\m{x}z)}{z}=x_{i}.
\end{eqnarray*}
(Note that only the boundary condition depends on $i$). Describe the algebraic and analytic properties of the flow $\Lambda_{N}(\m{x})$.
\end{prob}
We call the function $\Lambda_{N}(\m{x})$ \emph{the $\Sigma_{N+1}$-superflow}. The solution of this problem might reveal deeper connections with various objects encountered in algebraic geometry and number theory. More generally, we may pose the same problem not only for $\Sigma_{N+1}\subset{\rm GL}(\mathcal{A})$, but also for any finite (sufficiently large, as compared to the degree of representation) subgroup of the linear group. The results in the Subsecion \ref{sub-symmetry} show that this might be interpreted as symmetry identities for certain rational functions in the quotient rings.
\begin{defin}
Let $N\in\mathbb{N}$, $N\geq 2$, and $\Gamma\hookrightarrow{\rm GL}(\mathbb{C}^{N})$ be an exact representation of a finite group, and we identify $\Gamma$ with the image. We call the flow $\phi(\m{x})$ \emph{the $\Gamma$-superflow}, if there exists a vector field $\mathbf{Q}(\m{x})\neq (0,0,\ldots,0)$ whose components are $2$-homogenic rational functions and which is exactly the vector field of the flow $\phi(\m{x})$, such that (\ref{kappa}) is satisfied for all $\gamma\in\Gamma$, every other vector $\mathbf{Q}'$ which satisfies (\ref{kappa}) for all $\gamma\in\Gamma$ is either a scalar multiple of $\mathbf{Q}$, or  its degree of a common denominator is higher. Thus, $\mathbf{Q}$ is uniquely defined up to conjugation with a linear map $\m{x}\mapsto z\m{x}$.
\end{defin} 
Note that we do not require the components of $\mathbf{Q}(\m{x})$, which is the vector field of the flow $\phi$,  to be quadratic forms, they are just $2$-homogenic rational functions. The main reason is obvious - the flow $\phi$ might possess lots of symmetries while $\ell^{-1}\circ\phi\circ\ell$ might not, even if we confine to special $1$-BIRs $\ell$. Another reason for this is that \emph{a priori} there is no reason why for every vector field $\mathbf{Q}$ there must exists an $N$-dimensional $1$-BIR $\ell$ such that the vector field of $\ell^{-1}\circ\phi\circ\ell$ has only quadratic forms as its components. We know that even in a $2$-dimensional case the reduction is rather complicated, and there exists a particular obstruction where such $\ell$ does not exist (\cite{alkauskas2}, Section 4).
\begin{prob}
Describe all groups $\Gamma$ inside ${\rm GL}(\mathbb{C}^{N})$ for which there exists a $\Gamma$-superflow. Describe the properties of such a flow $\phi(\m{x})=\phi_{\Gamma}(\m{x})$ algebraically and analytically. 
\end{prob}
For example, the group $S_{4}$ has two non-equivalent exact irreducible $3$-dimensional representations. The first is constructed in the beginning of this subsection, while the second one, call it $\Sigma'_{4}$, arises as the group of rotations of a cube, and it is given by the embedding
\begin{eqnarray*}
(12)\mapsto\begin{pmatrix}
0 & 1 & 0\\
1 & 0 & 0\\
0 & 0 & -1 
\end{pmatrix},\quad
(13)\mapsto\begin{pmatrix}
0 & 0 & 1\\
0 & -1 &0\\
1 & 0 & 0 
\end{pmatrix},\quad
(14)\mapsto\begin{pmatrix}
-1 & 0 & 0\\
0 & 0 & 1\\
0 & 1 & 0 
\end{pmatrix}.
\end{eqnarray*}
The Young diagram of this representation is dual to the Young diagram of the first representation. Unfortunately, the vector $\mathbf{Q}$ of quadratic forms which satisfies (\ref{kappa}) for all $\gamma\in\Sigma'_{4}$, is trivial. But there exists a unique, up to scalar multiple, $\mathbf{Q}$, which is invariant under even permutations, that is, the image of the group $A_{4}$. This vector field is given by
\begin{eqnarray}
\mathbf{Q}(\m{x})=(yz,xz,xy).
\label{pelican}
\end{eqnarray}
To check the invariance we note that this group contains the order $4$ subgroup consisting of the elements $I$, ${\rm diag}(-1,-1,1)$, ${\rm diag}(-1,1,-1)$, and ${\rm diag}(1,-1,-1)$ (the fourth Klein group). The invariance under conjugation with these transformations are immediate, and we only need to check invariance under one element of order three, say, the image of $(12)(13)=(123)$. Nevertheless, this vector field is also invariant under the matrix
\begin{eqnarray*}
\begin{pmatrix}
0 & 1 & 0\\
1 & 0 & 0\\
0 & 0 & 1 
\end{pmatrix},
\end{eqnarray*}
and adjoining this matrix to tha image of $A_{4}$ we thus obtain the $S_{4}$-superflow, which is linearly conjugate to the flow $\Lambda_{3}$, as given above. So, for the contragradient representation $\Sigma'_{4}$ we obtain no new superflow, if the vector field is given by the collection of quadratic forms. Such negative outcome happens for majority of finite subgroups of the full linear group, provided these groups  are large enough or have a particular structure. For example, already one condition (\ref{kappa}) for $\gamma=-I$ forces $\mathbf{Q}$ to be trivial. However, this is not the end of the story in the case of $\Sigma'_{4}$. One can check directly the the vector field
\begin{eqnarray*}
\mathbf{Q}(\m{x})=
\frac{y^3z-yz^3}{x^2+y^2+z^2}\bl\frac{z^3x-zx^3}{x^2+y^2+z^2}\bl\frac{x^3y-xy^3}{x^2+y^2+z^2}
\end{eqnarray*}
is invariant under conjugation with all matrices from $\Sigma'_{4}$. Up to the constant factor, this is the only vector field whose common numerator is of degree at most $2$; so we get a $\Sigma'_{4}$-superflow, essentially different form $\Sigma_{4}$-superflow. We leave this very promising side of investigations of the projective translation equation for the future. Concerning the vector field (\ref{pelican}), we could perform similar analysis as in the case of $\Lambda=\Lambda_{2}$. In fact, few initial steps are already contained in the literature - this vector field lies exactly in the intersection of P\'{o}lya urns (all components are monomials) and projective flows (they are $2$-homogenic). This urn model is called \emph{the pelican's urn}, it can be analytically solved in terms of Jacobian elliptic functions, and the combinatoric and analytic theory is very well developed and is profound \cite{dumont1, dumont3, dumont2, fdp, schett, viennot}. So, our contribution to this theory is the remark that this flow has a $24$-fold symmetry, and the next step would be to rewrite the analytic and symmetry results of the Subsections \ref{sub:ann} and \ref{sub-symmetry} in the case of the vector field $(\ref{pelican})$ purely algebraically. This seems to be very promising.

\subsection{Topology. Continuous flows on the sphere $\mathbb{S}^{k}$, $k\geq 3$.}Consider the functional equation (\ref{funkk}) over $\mathbb{R}$ in dimension $k\geq 3$, and suppose that a solution extends as a continuous function $\phi:\widehat{\mathbb{R}^{k}}\mapsto\widehat{\mathbb{R}^{k}}$, where the latter is a single point compactification of $\mathbb{R}^{k}$, this point being denoted by $``\infty"$; we make no further regularity assumptions on $\phi$.
\begin{prob}Is the following true? If the above assumption holds, then such $\phi$ is either given by $\phi(\m{x})=\m{x}$, $\phi(\m{x})=\m{0}$, $\phi(\m{x})=\infty$, or there exists a $1$-homogenic continuous function $\ell:\mathbb{R}^{k}\mapsto\mathbb{R}^{k}$ such that $\ell^{-1}\circ\phi\circ\ell$ is either $\phi_{1}$ or $\phi_{\infty}$, where the $j$-th coordinates of $\phi_{1}$ and $\phi_{\infty}$ are given, respectively, for $j=1,\ldots,k$, by
\begin{eqnarray*}
(\phi_{1}(\m{x}))_{j}=\frac{\sum\limits_{i=1}^{k}x^{2}_{i}+k\cdot x_{j}}
{\sum\limits_{i=1}^{k}(x_{i}+1)^2},\quad
(\phi_{\infty}(\m{x}))_{j}=d_{j}\Big{(}\sum\limits_{i=1}^{k}x_{i}\Big{)}^{2}+x_{j};
\end{eqnarray*}
here $\m{d}=(d_{1},d_{2},\ldots,d_{k})$ is a fixed in advance non-zero vector such that $\sum_{i=1}^{k}d_{i}=0$.
\end{prob}
The main result in \cite{alkauskas1} claims that this is true when $k=2$.
\subsection{Birational geometry. Rational flows in dimension $k\geq 3$.}As mentioned in (\cite{alkauskas2}, Subsection 5.6), in order to understand the algebro-geometric nature of projective flows, one needs to investigate the structure of rational solutions of (\ref{funkk}) in dimension $k\geq 3$ over $\mathbb{C}$ or any other algebraically closed field of characteristic $0$ (this is the easiest case of the theory). In particular, we pose
\begin{prob}Let $\phi(x,y,w)=u(x,y,w)\bl v(x,y,w)\bl t(x,y,w)\neq x\bl y\bl w$ be a triplet of rational functions such that $\phi$ satisfies (\ref{funkk}) and the boundary condition (\ref{bound}). Is it true that then there exists a $1$-homogenic birational transformation of $\mathbb{C}^{3}$, call it $\ell$, such that
\begin{eqnarray*}
\ell^{-1}\circ\phi\circ\ell(x,y,w)=x(w+1)^{N-1}\bl\frac{y}{w+1}\bl\frac{w}{w+1}
\end{eqnarray*}
for a certain non-negative integer $N$?
\end{prob}
This is certainly a hard problem, unless some more powerful techniques than the ones used in \cite{alkauskas2} are applied. Moreover, this problem, if solved, would exhibit only basic invariants of a $k$-dimensional rational projective flows. For example, even in dimension $2$ there are many interesting characteristics of projective flows, such as surface (compact or non-compact) on which a flow is naturally defined; symmetries; limits under conjugation; orbits; fixed points; homotopy of flows, and many more \cite{alkauskas3}.
As a continuation of the previous problem and as a generalizaion of our Theorem \ref{thm1}, we pose
\begin{prob}Let $k\geq 3$. Classify all $k$-dimensional smooth projective flows such that they are non-singular, i.e. satisfy (\ref{bound}), they have rational vector field, and are unramified.
\end{prob}
\subsection{Quasi-rational flows of arbitrary genus}
\label{sub-qf}
The results of Subsections \ref{corro} and \ref{sub:rat} suggest that projective rational flows can be generalized purely algebraically without appeal to elliptic, abelian or other transcendental functions. We will define the notion of \emph{$2$-dimensional projective quasi-rational flow of genus $g$}, quasi-flow in short. Suppose, given a pair of $2$-homogenic rational functions $\varpi(x,y)$ and $\varrho(x,y)$. Suppose there exists a positive integer $N$ such that the ODE
\begin{eqnarray}
Nf(x)\varrho(x,1)
+f'(x)[\varpi(x,1)-x\varrho(x,1)]=0
\label{qua}
\end{eqnarray}
has a rational solution $f(x)$. Let $N$ be the smallest such integer, and let $\mathscr{W}(x,y)=y^{N}f(\frac{x}{y})=P(x,y)Q(x,y)^{-1}$, where $P$ and $Q$ are homogenic polynomials of degrees $N+d$ and $d$ for some $d\geq 0$, and their ratio is irreducible over $\mathbb{C}$. The orbits of the flow $\phi=u\bl v$ defined by (\ref{g-parteq}) and below are then the curves $\mathscr{W}(x,y)=c$, $c\in\mathbb{C}\cup\{\infty\}$. Of course, $\phi$ itself may turn out to be multi-valued analytic function with many branching points. Suppose the genus of the projective curve $P(x,y)=z^{N}Q(x,y)$ is $g\geq 0$. 
\begin{defin}Given a pair of $2$-homogenic rational functions $\varpi\bl\varrho$ such that the above assumption concerning ODE (\ref{qua}) does hold, $N$ being the smallest such positive integer. The rational function $U(A,B;x,y)$  is called \emph{a pre-quasi-flow of genus $g$} with vector field $\varpi\bl\varrho$, if
\begin{itemize}
\item[(i)]$U$ is $1$-homogenic:
\begin{eqnarray*}
U(zA,zB;zx,zy)=zU(A,B;x,y);
\end{eqnarray*}
\item[(ii)] The function $U$ satisfies the PDE in the quotient ring, i.e. the following holds:
\begin{eqnarray}
U_{A}\varpi(A,B)+U_{B}\varrho(A,B)+
U_{x}\varpi(x,y)+U_{y}\varrho(x,y)\nonumber\\
\equiv 0\text{ mod }\Big{(}P(A,B)Q(x,y)-P(x,y)Q(A,B)\Big{)}.
\label{u-quasi}
\end{eqnarray} 
\end{itemize}
\end{defin}
To define a \emph{quasi-flow}, not just a pre-quasi-flow, we need to specify the boundary conditions exactly, to define a pair of such functions $U(A,B;x,y)\bl V(A,B;x,y)$, both of which satisfy the above definition but which are subject to different boundary conditions. For this purpose we need to delve deeper into a certain system of ODE's. This topic is the central topic of the prospective paper \cite{alkauskas4}. Here we confine with the example $\mathscr{T}$ given in the Subsection \ref{sub:rat} and the examples to follow now. For example, it turns out that any rational flow gives rise to a quasi-flow of genus $0$. We will show this in case of the canonical flow $\phi_{N}$, see the Subsection \ref{back}; the general case will be treated in \cite{alkauskas3}. Since the vector field of the flow $\phi_{N}$ is $(N-1)xy\bl(-y^2)$, and $\mathscr{W}(x,y)=xy^{N-1}$, the general strategy of the Subsection \ref{sub:rat} shows that we need to solve the  system
\begin{eqnarray*}
\left\{\begin{array}{c@{\qquad}l}
a(u)b(u)^{N-1}=1,\\
a'(u)=-(N-1)a(u)b(u),\\
b'(u)=b^{2}(u).
\end{array}\right.
\end{eqnarray*}
So, we can choose $a=-u^{N-1}$, $b=-u^{-1}$. In our case $\v=(xy^{N-1})^{\frac{1}{N}}$. So, let $A=a(\v)\v=-xy^{N-1}$, $B=b(\v)\v=-1$.
Thus,
\begin{eqnarray*}
x(y+1)^{N-1}=-\frac{A(y-B)^{N-1}}{y^{N-1}}=U(A,B;x,y);\quad 
\frac{y}{y+1}=\frac{yB}{B-y}=V(A,B;x,y).
\end{eqnarray*}
We check that now the l.h.s. of (\ref{u-quasi}) in both cases is identically $0$, so the congruence holds all the more. The bondary conditions read as 
\begin{eqnarray*}
\lim\limits_{z\rightarrow 0}\frac{U(-z^{N}xy^{N-1},-1;xz,yz)}{z}=x,\quad \lim\limits_{z\rightarrow 0}\frac{V(-z^{N}xy^{N-1},-1;xz,yz)}{z}=y.
\end{eqnarray*}
Note that the function $U$, even if we enforce three conditions - (i), (ii), and the above boundary condition - is not uniquely defined.  For example, all three are also satisfied by
\begin{eqnarray*}
x(y+1)^{N-1}=\frac{x(B-y)^{N-1}}{B^{N-1}}=\hat{U}(A,B;x,y)=-U(x,y;A,B).
\end{eqnarray*}
Note also that $V(A,B;x,y)=-V(x,y;A,B)$. Thus, rational flows give rise to quasi-flows of genus $0$. The converse is not true. We will show this on the example of the flow $\phi^{{\rm e}}$, as given by the Proposition \ref{prop1}. Let therefore $N=1$, $\varpi(x,y)=(x^2-y^2)/2$, $\varrho(x,y)=(y^2-x^2)/2$. Then $f(x)=x+1$, and $P(x,y)=x+y$, $Q(x,y)=1$. The curve $x+y=z$ is a line and thus is of genus $0$. The system
\begin{eqnarray*}
\left\{\begin{array}{c@{\qquad}l}
a(u)+b(u)=1,\\
2a'(u)=b^{2}(u)-a^{2}(u),\\
2b'(u)=a^{2}(u)-b^{2}(u),
\end{array}\right.
\end{eqnarray*}
has a solution $a(u)=\frac{1}{2}-e^{-u}$, $b(u)=\frac{1}{2}+e^{-u}$. In this case $\v=x+y$, so let 
\begin{eqnarray*}
A&=&a(\v)\v=(x+y)\Big{(}\frac{1}{2}-e^{-(x+y)}\Big{)},\\
B&=&b(\v)\v=(x+y)\Big{(}\frac{1}{2}+e^{-(x+y)}\Big{)}.
\end{eqnarray*}
So, according to Proposition \ref{prop1},
\begin{eqnarray*}
u^{{\rm e}}(x,y)=\frac{1}{2}\big{(}(x-y)e^{x+y}+x+y\big{)}=
\frac{x^2-y^2}{B-A}+\frac{x+y}{2}=U(A,B;x,y)=V(A,B;y,x),
\end{eqnarray*}
and
\begin{eqnarray*}
U_{A}\varpi(A,B)+U_{B}\varrho(A,B)+
U_{x}\varpi(x,y)+U_{y}\varrho(x,y)=\frac{(x^2-y^2)(x+y-A-B)}{B-A},
\end{eqnarray*}
and so (\ref{u-quasi}) holds. The boundary conditions now read as 
\begin{eqnarray*}
\lim\limits_{z\rightarrow 0}\frac{U(-\frac{1}{2}(x+y)z,\frac{3}{2}(x+y)z;xz,yz)}{z}=x,\quad
 \lim\limits_{z\rightarrow 0}\frac{V(-\frac{1}{2}(x+y)z,\frac{3}{2}(x+y)z;xz,yz)}{z}=y;
\end{eqnarray*}
(since both functions are $1$-homogenic, we have an identity and not just the limit). Nevertheless, this quasi-flow arises from the flow $\phi^{{\rm e}}$ which is  not rational, though the orbits are lines. From the other side, given a projective curve $\mathscr{W}(x,y)=z^{N}$. There are infinitely many pairs of $2$-homogenic rational functions $\varpi$ and $\varrho$ such that
\begin{eqnarray*}
\mathscr{W}_{x}(x,y)\varpi(x,y)+\mathscr{W}_{y}(x,y)\varrho(x,y)=0.
\end{eqnarray*}
So, there are potentially many flows having the same orbits. 
\begin{prob}Develop a theory of quasi-flows. In particular, find all quasi-flows with orbits given by $\mathscr{W}(x,y)=c$ for any curve of arbitrary genus. Find a condition on $\varpi$ and $\varrho$ which, if satisfied, leads to a quasi-flow.
\end{prob}
As was proved in \cite{alkauskas2}, the pair $\varpi\bl\varrho$ leads to a rational flow if and only if
\begin{eqnarray*}
\frac{y\varpi_{y}(x,y)-x\varrho_{y}(x,y)}{y\varpi_{x}(x,y)-x\varrho_{x}(x,y)}=\frac{ax+by}{cx+dy}
\end{eqnarray*}
for certain $a,b,c,d\in\mathbb{C}$ satisfying $a\neq d$, and
\begin{eqnarray*}
\frac{(a+d)^2-4bc}{(a-d)^2}=M^2,\text{ for a certain }M\in\mathbb{N}.
\end{eqnarray*}
In this case, the integer $M$ is exactly the level of a flow. We therefore ask for a similar criterio for quasi-flow.
\subsection{Projective flows over finite fields.}For simplicity, consider a finite field $\mathbb{F}_{p}$, $p$ being a prime. Let $\widehat{\mathbb{F}}_{p}=\mathbb{F}_{p}\cup\{\infty\}$, and extend the sum and the multiplication operations, initially defined on $\mathbb{F}_{p}$, to $\widehat{\mathbb{F}}_{p}$ using the natural definition: that is, if $a\in\mathbb{F}_{p}\setminus\{0\}$, then $a\cdot\infty=\infty$, $a+\infty=\infty$, $\frac{a}{\infty}=0$, $0+\infty=\infty$, $\infty\cdot\infty=\infty$,
but $0\cdot\infty$, $\infty+\infty$ are undefined.
We will now find all $1-$dimenional projective flows over finite fields.
Suppose, $f:\widehat{\mathbb{F}}_{p}\mapsto\widehat{\mathbb{F}}_{p}$ satisfies
\begin{eqnarray*}
(1-z)f(x)=f\Big{(}f(xz)\frac{1-z}{z}\Big{)},\quad x,z\in\widehat{\mathbb{F}}_{p}.
\end{eqnarray*}
Then $f(x)\equiv 0$, $f(x)\equiv\infty$, or there exists $a\in\mathbb{F}_{p}$, such that
$f(x)=x(ax+1)^{-1}$. Indeed, assume $f(\infty)=b$, $b\neq 0,\infty$. Then substitution $x=\infty$ into the functional equation, for $z\neq 0$, gives
$(1-z)b=f(b(1-z)z^{-1})$, and the answer follows for $a=b^{-1}$. The cases $f(\infty)=0$ or $\infty$ are dealt similarly.
So, there are exactly $p+2$ projective $1$-dimensional flows over the finite field $\mathbb{F}_{p}$, $p$ of them are non-singular. Now consider the dimension $2$. For further convenience we note that (\ref{funkk}) implies the identity
\begin{eqnarray}
n\underbrace{\phi\circ\cdots\circ\phi}\limits_{n}(\m{x})=\phi(n\m{x}),\quad n\in\mathbb{N};
\label{iterr}
\end{eqnarray}
this is valid in any dimension over any field  provided that $n$ is not a multiple of the characteristic of this field. Returning to finite fields and dimension $2$, we encounter the interesting problem to investigate where the projective flow is defined. For example, the flows
\begin{eqnarray*}
(x-y)^2+x\bl (x-y)^2+y,\quad \frac{x}{x+1}\bl\frac{y}{y+1},\quad
\frac{x}{x+y+1}\bl\frac{y}{x+y+1}
\end{eqnarray*}
are defined and take values on spaces $\mathbb{F}_{p}^{2}$, $(\widehat{\mathbb{F}}_{p})^2$ and $P^{2}(\mathbb{F}_{p})$, of cardinalities $p^2$, $(p+1)^2$, $p^2+p+1$, respectively. In fact, these maps are automorphisms of corresponding spaces (in this case this simply means they are bijections). On the other hand, let $p$ be an odd prime. The nature of the flow
\begin{eqnarray*}
\phi_{p}(x,y)=\frac{x^2+y^2+2x}{(x+1)^2+(y+1)^2}\bl \frac{x^2+y^2+2y}{(x+1)^2+(y+1)^2}
\end{eqnarray*}
depends on the arithmetic of $p$. If $\left(\frac{-1}{p}\right)=-1$, this flow is the bijection from $\mathbb{F}_{p}^{2}\cup\{\infty\}$ onto itself, where $\phi(-1,-1)=\infty$, $\phi(\infty)=1\bl 1$; this space (``sphere") is of cardinality $p^2+1$. Consider now the case $\left(\frac{-1}{p}\right)=1$; for example, take $p=5$. For $(x,y)\in\mathbb{F}_{5}^2$, the function $\phi$ is defined for $16$ pairs, for $9$ pairs it is undefined, and it takes also $16$ and leaves $9$ values. It appears that the only meaningful completion is to add $11$ additional points. For example, since  $\phi(3\bl 1)$ is undefined, let $\phi(3\bl 1)=\eta$. Further, by the projective flow property (\ref{iterr}) $\phi(3\bl 1)=\phi\circ\phi(4\bl 3)=2^{-1}\phi(2(4\bl 3))=2^{-1}\phi(3\bl 1)$, so $2\eta=\eta$. So, $\eta$ can be formally denoted by $\infty\bl 0$. Further, let $\phi(0\bl 1)=\alpha$. Then $\phi(\alpha)=\phi\circ\phi(0\bl 1)=2^{-1}\phi(2(0\bl 1))=2^{-1}\phi(0\bl 2)=3\phi(0\bl 2)$. More calculations confirm that $\alpha$ can be formally denoted by $1\bl\infty$. Once again, $\phi(\eta)=\phi\circ\phi(3\bl 1)=2^{-1}\phi(2(3\bl 1))=2^{-1}\phi(1\bl 2)=2^{-1}(4\bl 3)=2\bl 4$. Also, $\phi\circ\phi\circ\phi(3\bl 1)=3^{-1}\phi(3(3\bl 1))=3^{-1}\phi(4\bl 3)=1\bl 2$. By direct similar calculations, we thus complete the  Table \ref{table3}.
\begin{table}[h]
\begin{tabular}{c|  c c c c c c}
$x\backslash y$& $0$ & $1$ & $2$ & $3$ & $4$ & $\infty$\\
  \hline
0 & $0\bl 0$ & $1\bl\infty$ & $\infty\bl 1$ & $2\bl0$ & $1\bl4$ & $4\bl 2$\\
1 & $\infty\bl 2$ & $3\bl3$ & $4\bl 3$ & $0\bl\infty$ & $1\bl0$ & $\infty\bl 3$\\
2 & $3\bl\infty$ & $3\bl 4$ & $4\bl4$ & $\infty\bl 4$ & $1\bl2$ & $3\bl 0$\\
3 & $0\bl 2$ & $\infty\bl 0$ & $2\bl\infty$ & $2\bl2$ & $1\bl3$ & $2\bl 3$\\
4 & $4\bl1$ & $0\bl1$ & $2\bl1$ & $3\bl1$ & $\infty\bl\infty$ & $0\bl 4$\\
$\infty$ & $2\bl 4$ & $3\bl 2$ & $4\bl\infty$ & $4\bl 0$ & $0\bl 3$ & $1\bl 1$\\
\hline
\end{tabular}
\caption{The flow $\phi_{p}$ over the finite field $\mathbb{F}_{5}$}
\label{table3}
\end{table}
The symbols $1\bl\infty$ and $0\bl\infty$, for example, obey the rules $2(1\bl\infty)=2\bl\infty$ and $3(0\bl\infty)=0\bl\infty$.  So, the case $\left(\frac{-1}{p}\right)=1$ corresponds a bijection of $(\widehat{\mathbb{F}}_{p})^{2}$ onto itself and thus to torus rather than a sphere. On the other hand, the analysis of the flow
\begin{eqnarray*}
\frac{x(x^2+y^2)}{x^2y+xy^2+x^2+y^2}\bl\frac{x(x^2+y^2)}{x^2y+xy^2+x^2+y^2}
\end{eqnarray*}
leads to investigations of the arithmetic of genus $0$ cubic $x^2y+xy^2+x^2+y^2=0$ over $\mathbb{F}_{p}$. Further, it is obvious that the number of projective flows over a finite field in any dimension is finite. Finally, note that the exponential flow $\phi^{\rm exp}$ has no analogue in the finite field setting; for example, the flow $x2^{y}\bl y$ is not property defined on $(\mathbb{F}_{p})^2$ for $p$ odd. So, we formulate the following
\begin{prob}Let $k\geq 2$. Describe all projective flows over $(\mathbb{F}_{p})^{k}$. What is the total number of flows? How they distribute among different completions of  $(\mathbb{F}_{p})^{k}$? Do all these flows arise as reductions {\rm mod }$p$ of rational flows over $\mathbb{Q}$ (like $\phi_{p}$ above)?
\end{prob}


\begin{thebibliography}{3000}
\bibitem[Ac66]{aczel}{\sc J. Acz\'{e}l}, {\it Lectures on functional equations and their applications}, Mathematics in Science and Engineering, Vol. 19 Academic Press, New York-London 1966.

\bibitem[Ah48]{ahiezer}{\sc N.I. Ahiezer}, N. I. {\it \`{E}lementy teorii \`{e}llipti\u{c}eskih funkci\u{i}}. (Russian) [{\it Elements of the Theory of Elliptic Functions}] Gosudarstv. Izdat. Tehn.-Teor. Lit., Moscow-Leningrad, 1948. 291 pp.

\bibitem[Al10]{alkauskas1} {\sc G. Alkauskas}, Multi-variable translation equation which arises from homothety. {\it Aequationes Math.}
{\bf 80} (3) (2010), 335--350; {\tt arxiv.org/abs/0911.1513}.

\bibitem[Al12]{alkauskas2}{\sc G. Alkauskas}, The projective translation equation and rational plane flows. I (submitted); {\tt arxiv.org/abs/1201.0894}.

\bibitem[Alpr1]{alkauskas3}{\sc G. Alkauskas}, The projective translation equation and rational plane flows. II  (in preparation)

\bibitem[Alpr2]{alkauskas4}{\sc G. Alkauskas}, The projective translation equation: rational vector fields and quasi-flows (tentative title).

\bibitem[Di90]{dixon}{\sc A.C. Dixon}, On the doubly periodic functions arising out of the curve $x^3 + y^3 - 3\alpha xy = 1$, {\it Quart. J.} {\bf XXIV} (1890). 167--233.


\bibitem[Du79]{dumont1}{\sc D. Dumont}, A combinatorial interpretation for the Schett recurrence on the Jacobian elliptic functions. {\it Math. Comp.} {\bf 33} (1979), 1293--1297.


\bibitem[Du81]{dumont3}{\sc D. Dumont}, Une approche combinatoire des fonctions elliptiques de Jacobi. {\it Adv. Math.} {\bf 1} (1981), 1--39.

\bibitem[Du86]{dumont2}{\sc D. Dumont}, Grammaires de William Chen et d\'{e}rivations dans les arbres et arborescences, {\it S\'{e}m. Lothar. Combin.} {\bf 37} (1996), Art. B37a, 21 pp. (electronic).

\bibitem[BAT]{erdelyi}{\sc A. Erd\'{e}lyi, W. Magnus, F. Oberhettinger, F.G. Tricomi, Francesco G.}, {\it Higher transcendental functions. Vol. I. Based on notes left by Harry Bateman.} With a preface by Mina Rees. With a foreword by E. C. Watson. Reprint of the 1953 original. Robert E. Krieger Publishing Co., Inc., Melbourne, Fla., 1981.

\bibitem[BF10]{flajolet-b}{\sc R. Bacher, Ph. Flajolet}, Pseudo-factorials, elliptic functions, and continued fractions. {\it Ramanujan J.}
{\bf 21} (1)(2010), 71--97; {\tt arxiv.org/abs/0901.1379}.

\bibitem[BF11]{flajolet-bl}{\sc P. Blasiak, Ph. Flajolet}, Combinatorial models of creation-annihilation. {\it S\'{e}m. Lothar. Combin.} {\bf 65} (2011), Art. B65c; {\tt arxiv.org/abs/1010.0354}.

\bibitem[FGP05]{flajolet-g-p}{\sc Ph. Flajolet, J. Gabarr\'{o}, H. Pekari}, Analytic urns.
{\it Ann. Probab.} {\bf 33} (3) (2005), 1200--1233; {\tt arxiv.org/abs/math/0407098}.

\bibitem[FCF05]{flajolet-c}{\sc E.van Fossen Conrad, Ph. Flajolet}, The Fermat cubic, elliptic functions, continued fractions, and a combinatorial excursion.
{\it S\'{e}m. Lothar. Combin.} {\bf 54} (2005/07), Art. B54g; {\tt arxiv.org/abs/math/0507268}.

\bibitem[FDP06]{fdp}{\sc Ph. Flajolet, Ph. Dumas, V. Puyhaubert}, Some exactly solvable models of urn process theory. {\it Fourth Colloquium on Mathematics and Computer Science Algorithms, Trees, Combinatorics and Probabilities, Discrete Math. Theor. Comput. Sci. Proc., AG, Assoc. Discrete Math. Theor. Comput. Sci., Nancy}, (2006), 59--118.

\bibitem[FR08]{r-r1} {\sc H. Fripertinger, L. Reich}, The formal translation equation and formal cocycle equations for iteration groups of type I. {\it Aequationes Math.}
{\bf 76} (1-2) (2008), 54--91.

\bibitem[FR10]{r-r2} {\sc H. Fripertinger, L. Reich}, The formal translation equation for iteration groups of type II. {\it Aequationes Math.}
{\bf 79} (1-2) (2010), 111--156.

\bibitem[Kn92]{knapp}{\sc A.W. Knapp}, {\it Elliptic curves}, Mathematical Notes, 40. Princeton University Press, Princeton, NJ, 1992.

\bibitem[LS87]{shabat}{\sc M.A. Lavrent'ev, B.V. Shabat},  {\it Metody teorii funktsi\u{i} kompleksnogo peremennogo}. (Russian) [{\it Methods of the theory of functions in a complex variable}] Fifth edition. Nauka, Moscow, 1987. 688 pp.

\bibitem[Sch76]{schett} {\sc A. Schett}, Properties of the Taylor series expansion coefficients of the Jacobian elliptic functions. {\it Math. Comp.}
{\bf 30} (1976), 143--147.

\bibitem[Mo73]{moszner1}{\sc Z. Moszner}, The translation equation and its application. {\it Demonstratio Math.} {\bf 6} (1973), 309--327.

\bibitem[Mo95]{moszner2}{\sc Z. Moszner}, General theory of the translation equation. {\it Aequationes Math.}  {\bf 50}(1-2)  (1995), 17--37.

\bibitem[Vie80]{viennot}{\sc G. Viennot}, Une interpretation combinatoire des coefficients des d\'{e}veloppments en s\'{e}rie enti\'{e}re des fonctions elliptiques de Jacobi. 
{\it  J. Combin. Theory Ser. B.}  {\bf 29} (1980), 121--133.

\end{thebibliography}
\end{document}